\documentclass{amsart}
\usepackage{amssymb,color}
\usepackage{eucal}
\theoremstyle{plain}
\newtheorem{theorem}{Theorem}
\newtheorem{corollary}{Corollary}

\newtheorem{lemma}{Lemma}

\theoremstyle{definition}
\newtheorem{definition}{Definition}

\theoremstyle{remark}

\newtheorem{remark}{Remark}
\numberwithin{equation}{section}
\newcommand{\e}{\epsilon}
\newcommand{\s}{\sigma}

\newcommand{\R}{\mathbb R}
\newcommand{\N}{\mathbb N}

\newcommand{\Rn}{\mathbb R^n}

\newcommand{\rst}[1]{\ensuremath{{\mathbin\upharpoonright}%
\raise-.5ex\hbox{$#1$}}}

\def\bb{\begin{equation}} \def\ee{\end{equation}}

\def \p  {\partial}

\def \om {\Omega}
\def \o {\omega}
\begin{document}
%
%
%
%
%
\title[Observability inequalities and measurable sets]{Observability inequalities and measurable sets}

\author{J. Apraiz}
\address[J. Apraiz]{Universidad del Pa{\'\i}s Vasco/Euskal Herriko
Unibertsitatea\\Departamento de Matem\'atica Aplicada\\Escuela Universitaria Polit\'ecnica de Donostia-San Sebasti\'an\\Plaza de Europa 1\\20018 Donostia-San Sebasti\'an, Spain.}
\email{jone.apraiz@ehu.es}

\author{L. Escauriaza}
\address[L. Escauriaza]{Universidad del Pa{\'\i}s Vasco/Euskal Herriko
Unibertsitatea\\Dpto. de Matem\'aticas\\Apto. 644, 48080 Bilbao, Spain.}
\email{luis.escauriaza@ehu.es}
\thanks{The first two authors are supported  by Ministerio de Ciencia e Innovaci\'on grants, MTM2004-03029 and MTM2011-2405.}
\thanks{The last two authors are
supported by
the National Natural Science Foundation of China under grants
11161130003 and 11171264 and partially by National Basis Research
 Program of China (973 Program) under grant 2011CB808002.}
\author{G. Wang}
\address[G. Wang]{Department of Mathematics and Statistics, Wuhan University, Wuhan, China}
\email{wanggs62@yeah.net}

\author{C. Zhang}
\address[C. Zhang]{Department of Mathematics and Statistics, Wuhan University, Wuhan, China}
\email{zhangcansx@163.com}

\keywords{observability inequality, heat equation, measurable set, spectral inequality}
\subjclass{Primary: 35B37}
\begin{abstract}
This paper presents two observability inequalities for the heat equation over $\Omega\times (0,T)$. In the first one, the observation is from a subset of positive measure in $\Omega\times (0,T)$, while in the second, the observation is from a subset of positive surface measure on $\partial\Omega\times (0,T)$. It also proves the Lebeau-Robbiano spectral inequality when $\Omega$ is a  bounded  Lipschitz  and locally star-shaped  domain. Some applications for the above-mentioned observability inequalities are provided.
\end{abstract}
\maketitle
\section{Introduction}\label{S:1}
Let  $\Omega$ be a bounded Lipschitz  domain in $\mathbb{R}^n$  and $T$ be a fixed positive time. Consider the heat equation:
\begin{equation}\label{1.1}
\begin{cases}
\partial_tu-\Delta u=0,\ &\text{in}\ \Omega\times (0,T),\\
u=0,\  &\text{on}\ \partial\Omega\times(0,T),\\
u(0)=u_0,\ &\text{in}\ \Omega,
\end{cases}
\end{equation}
with $u_0$ in $L^2(\Omega)$. The solution of (\ref{1.1}) will be treated as either a function from $[0,T]$ to $L^2(\Omega)$ or a function of two variables $x$ and $t$.   Two important apriori estimates for the above equation are as follows:
\begin{equation}\label{1.2}
\|u(T)\|_{L^2(\Omega)}\leq N(\Omega,T,\mathcal{D})\int_{\mathcal{D}}|u(x,t)|\,dxdt,\;\;\mbox{for all}\;\; u_0\in L^2(\Omega),
\end{equation}
 where $\mathcal{D}$ is a subset of $\Omega\times(0,T)$, and
 \begin{equation}\label{1.3}
\|u(T)\|_{L^2(\Omega)}\leq N(\Omega,T,\mathcal{J})\int_{\mathcal{J}}|\tfrac{\partial }{\partial\nu}u(x,t)|\,d\sigma dt, \;\;\mbox{for all}\;\; u_0\in L^2(\Omega),
\end{equation}
where $\mathcal{J}$ is a subset of $\partial\Omega\times(0,T)$. Such apriori estimates are called observability inequalities.

In the case that  $\mathcal{D}=\omega\times (0,T)$ and $\mathcal{J}=\Gamma\times (0,T) $ with $\omega$ and $\Gamma$ accordingly  open and nonempty subsets of $\Omega$ and $\partial\Omega$, both inequalities (\ref{1.2}) and (\ref{1.3}) (where $\partial\Omega$ is smooth)  were  essentially first established, via the Lebeau-Robbiano spectral inequalities in \cite{G. LebeauL. Robbiano} (See also \cite{G. LebeauE. Zuazua, Miller2,Fernandez-CaraZuazua1}). These two estimates were set up to the linear parabolic equations (where $\partial\Omega$ is of class $C^2$), based on the Carleman inequality provided in \cite{FursikovOImanuvilov}. In the case when  $\mathcal{D}=\omega\times (0,T)$  and $\mathcal{J}=\Gamma\times (0,T) $ with $\omega$ and $\Gamma$ accordingly  subsets of positive measure and positive surface measure in $\Omega$ and $\partial\Omega$, both inequalities (\ref{1.2}) and (\ref{1.3}) were built up in \cite{ApraizEscauriaza1} with the help of a propagation of smallness estimate from measurable sets for real-analytic functions first established in \cite{Vessella} (See also Theorem \ref{T:3}).  For $\mathcal{D}=\omega\times E$,
with $\omega$ and $E$ accordingly  an open subset of $\Omega$ and a subset of positive measure in $(0,T)$, the inequality (\ref{1.2}) (with $\partial\Omega$ is smooth)  was proved in  \cite{gengshengwang1} with the aid of the Lebeau-Robbiano spectral inequality,  and
it was then verified for heat equations (where $\Omega$ is convex) with lower terms depending on the time variable, through a frequency function method in \cite{PhungWang1}. When $\mathcal{D}=\omega\times E$,
with $\omega$ and $E$ accordingly  subsets of positive measure in $\Omega$ and $(0,T)$,   the estimate (\ref{1.2}) (with $\partial\Omega$ is real-analytic) was obtained in  \cite{canzhang}.

The purpose of this study is to establish inequalities (\ref{1.2}) and (\ref{1.3}),  when $\mathcal{D}$ and $\mathcal{J}$ are  arbitrary subsets of positive measure and of positive surface measure in $\Omega\times (0,T)$ and $\partial\Omega\times (0,T)$ respectively. Such inequalities  not only are  mathematically interesting but also have important applications in the control theory  of the heat equation, such as the bang-bang control, the time optimal control, the null controllability over a measurable set and so on (See Section \ref{S:3} for the applications).

The starting point we choose here to prove the above-mentioned two inequalities is to assume that the Lebeau-Robbiano spectral inequality stands on $\Omega$. To introduce it, we write
\begin{equation*}
0<\lambda_1\le\lambda_2\le\dots\leq\lambda_j\le\cdots
\end{equation*}
for the eigenvalues of $-\Delta$ with the zero Dirichlet boundary condition over $\partial\Omega$, and  $\{e_j: j\ge1 \}$ for  the set of $L^2(\Omega)$-normalized eigenfunctions, i.e.,
\begin{equation*}
\begin{cases}
\Delta e_j+\lambda_je_j=0,\ &\text{in}\ \Omega,\\
e_j=0,\ &\text{on}\ \partial\Omega.
\end{cases}
\end{equation*}
For  $\lambda>0$ we define
\begin{equation*}
\mathcal E_\lambda  f=\sum_{\lambda_j\le\lambda}( f,e_j)\,e_j\quad \text{and}\quad \mathcal E_\lambda^\perp  f=\sum_{\lambda_j>\lambda}( f,e_j)\,e_j,
\end{equation*}
where
\begin{equation*}
( f,e_j)=\int_{\Omega} f\, e_j\,dx,\ \text{when}\   f\in L^2(\Omega),\ j\ge 1.
\end{equation*}
Throughout this paper the following notations  are effective:
\begin{equation*}
(f,g)=\int_{\Omega}fg\,dx\  \text{and}\ \|f\|_{L^2(\Omega)}=\left(f,f\right)^{\frac 12};
\end{equation*}
$\nu$  is the unit exterior normal vector to $\p\om$; $d\s$ is surface measure on $\partial\Omega$;
$B_R(x_0)$ stands for the ball centered at $x_0$ in $\R^n$ of radius $R$; $\triangle_R(x_0)$ denotes $B_R(x_0)\cap\partial\Omega$; $B_R=B_R(0)$, $\triangle_R=\triangle_R(0)$; for measurable sets $\omega\subset\Rn$ and $\mathcal{D}\subset\Rn\times (0,T)$, $|\omega|$ and $|\mathcal{D}|$ stand for the Lebesgue measures of the sets; for each measurable set $\mathcal{J}$ in $\partial\Omega\times (0,T)$, $|\mathcal{J}|$ denotes its surface measure on the lateral boundary of $\Omega\times \R$;  $\{e^{t\Delta} : t\geq 0\}$ is the semigroup generated by $\Delta$ with zero Dirichlet boundary condition over $\partial\Omega$. Consequently,  $e^{t\Delta}f$ is the solution of Equation \eqref{1.1} with the initial state $f$ in $L^2(\Omega)$.
  The Lebeau-Robbiano spectral inequality is as follows:
 \vspace{0.2cm}

 \emph{For each $0<R\le 1$, there is $N=N(\Omega,R)$, such that the inequality
\begin{equation}\label{E: desigualdadespectralocalizada}
\|\mathcal E_\lambda  f\|_{L^2(\Omega)}\le Ne^{N\sqrt\lambda}\|\mathcal E_\lambda  f\|_{L^2(B_R(x_0))}
\end{equation}
holds, when $B_{4R}(x_0)\subset\Omega$, $f\in L^2(\Omega)$ and $\lambda>0$.}
\bigskip

 To our best knowledge, the inequality \eqref{E: desigualdadespectralocalizada} has been proved under condition that  $\partial\Omega$ is at least $C^2$
 \cite{G. LebeauL. Robbiano,G. LebeauE. Zuazua,RousseauRobbiano2, {luqi}}. In the current work, we obtain this inequality when $\Omega$ is a  bounded Lipschitz and locally star-shaped domain in $\Rn$ (See Definitions \ref{D: Lipschitz domain} and \ref{D: contractable} in Section 3).  It can be observed from Section \ref{S:6}  that bounded $C^1$ domains,  Lipschitz polygons in the plane, Lipschitz polyhedrons in $\R^n$, with $n\geq 3$ and bounded convex domains in $\Rn$ are always bounded Lipschitz  and locally star-shaped (See Remarks \ref{R: algobastanteimportantene} and \ref{R: algoasombroso} in Section 3).

 Our main results related to the observability inequalities are stated as follows:

 \begin{theorem}\label{T:5carcajada} Suppose that  a bounded domain $\Omega$ verifies the condition \eqref{E: desigualdadespectralocalizada} and $T>0$. Let $x_0\in \Omega$ and $R\in (0,1]$ be such that
 $B_{4R}(x_0)\subset\Omega$. Then, for each measurable set $\mathcal D\subset B_R(x_0)\times (0,T)$ with $|\mathcal{D}|>0$,  there is a positive constant $B=B(\Omega, T, R, \mathcal{D})$, such that
\begin{equation}\label{E: observabilidadparabolica}
\|e^{T\Delta}f\|_{L^2(\Omega)}\le e^{B}\int_{\mathcal D}|e^{t\Delta}f(x)|\,dxdt,\;\;\mbox{when}\;\; f\in L^2(\Omega).
\end{equation}
\end{theorem}
\begin{theorem}\label{boundary observability}  Suppose that a bounded Lipschitz domain $\Omega$ verifies the condition \eqref{E: desigualdadespectralocalizada} and $T>0$.
Let $q_0\in \partial\Omega$ and $R\in (0,1]$ be such that $\triangle_{4R}(q_0)$ is real-analytic. Then, for each measurable set  $\mathcal J\subset \triangle_R(q_0)\times (0,T)$
with $|\mathcal{J}|>0$,  there is  a positive constant $B=B(\Omega, T, R, \mathcal{J})$, such that
\begin{equation}\label{E: observabilidadparabolica2frontera}
\|e^{T\Delta}f\|_{L^2(\Omega)}\le e^{B}\int_{\mathcal J}|\tfrac{\partial}{\partial\nu}\, e^{t\Delta}f(x)|\,d\s dt,\;\;\mbox{when}\;\; f\in L^2(\Omega).
\end{equation}
\end{theorem}

\noindent The definition of the real analyticity for $\triangle_{4R}(q_0)$ is given in Section 4 (See Definition~\ref{D: fromteralocalrealanalitica}).

\begin{theorem}\label{T: elteoremauquenuncaquiseescribir} Let $\Omega$ be a bounded Lispchitz and locally star-shaped domain in $\Rn$. Then, $\Omega$ verifies the condition \eqref{E: desigualdadespectralocalizada}.
\end{theorem}

It deserves mentioning that Theorem \ref{boundary observability} also holds when $\Omega$ is a Lipschitz polyhedron in $\Rn$ and $\mathcal J$ is a measurable subset  with positive surface measure of $\partial\Omega\times (0,T)$ (See the part $(ii)$ in  Remark~\ref{section4remark11}).

In Section  \ref{S:3} we explain some applications of the Theorems \ref{T:5carcajada} and \ref{boundary observability} in the control theory of the heat equation. In particular, the existence of $L^\infty(\mathcal D)$-interior  and $L^\infty(\mathcal J)$-boundary admissible controls, the uniqueness and bang-bang properties of the minimal $L^\infty$-norm  control and the uniqueness and bang-bang property for the optimal controls  associated to the first type and the second type of the time optimal control problems.

In this work we use the new strategy developed in \cite{PhungWang1} to prove parabolic observability inequalities: a mixing of ideas from \cite{Miller2}, the global interpolation inequalitiy in Theorems \ref{T: 2global 2-sphere 1-cylinder} and \ref{interpolation} and the telescoping series method. This new strategy can also be extended to more general parabolic evolutions with variable time-dependent second order coefficients and with unbounded lower order time-dependent coefficients. To do it one must prove the global interpolation inequalities in Theorems \ref{T: 2global 2-sphere 1-cylinder} and \ref{interpolation} for the corresponding parabolic evolutions. These can be derived in the more general setting from the Carleman inequalities in \cite{Escauriaza1,EscauriazaFernandez1,EscauriazaVega,Fernandez1,KochTataru} or from local versions of frequency function arguments \cite{EscauriazaFernandezVessella,PhungWang1}. Here we choose to derive the interpolation inequalities only for the heat equation and from the condition \eqref{E: desigualdadespectralocalizada} because it is technically less involved and helps to make the presentation of the basic ideas more clear.

The rest of the paper is organized as follows: Section \ref{S:2} proves Theorem \ref{T:5carcajada}; Section 3 shows Theorem \ref{T: elteoremauquenuncaquiseescribir}; Section \ref{S:5} verifies Theorem \ref{boundary observability}; Section \ref{S:3} presents the applications of
Theorem \ref{T:5carcajada} and Theorem \ref{boundary observability} and Section \ref{S:8} is an Appendix completeting some of the technical details in the work.

\section{Interior observability}\label{S:2}
 Throughout this section $\Omega$ denotes a bounded domain and $T$ is a positive time. First of all, we recall the following observability estimate or propagation of smallness inequality from measurable sets:
\begin{theorem}\label{T:3}
 Assume that $f:B_{2R}\subset\Rn\longrightarrow\R$ is real-analytic in $B_{2R}$ verifying
\begin{equation*}\label{E: condicion fundament}
|\partial^\alpha f(x)|\le \frac{M |\alpha|!}{(\rho R)^{|\alpha|}}\ ,\  \text{when}\ \ x\in B_{2R}, \ \alpha\in \N^n,
\end{equation*}
for some $M>0$ and $0<\rho\le 1$. Let $E\subset B_R$ be a measurable set with positive measure. Then, there are positive constants $N=N(\rho, |E|/|B_R|)$ and $\theta=\theta(\rho, |E|/|B_R|)$ such that
\begin{equation}\label{E: proagationdelapeque–ezdesdemedible}
\|f\|_{L^\infty(B_R)}\le N\left(\text{\rlap |{$\int_{E}$}}\,|f|\,dx\right)^{\theta}M^{1-\theta}.
\end{equation}
\end{theorem}
 The estimate \eqref{E: proagationdelapeque–ezdesdemedible} is first established in \cite{Vessella} (See also \cite{Nadirashvili2} and \cite{Nadirashvili} for other close results). The reader may find a simpler proof of  Theorem~\ref{T:3} in \cite[\S 3]{ApraizEscauriaza1}, the proof there was  built with ideas taken from \cite{Malinnikova}, \cite{Nadirashvili2} and  \cite{Vessella}.

Theorem \ref{T:3} and the condition \eqref{E: desigualdadespectralocalizada} imply the following:

\begin{theorem}\label{T: 1spectralinequality} Assume that  $\Omega$ verifies \eqref{E: desigualdadespectralocalizada}, $\omega$ is a subset of positive measure such that $\omega\subset B_{R}(x_0)$, with $B_{4R}(x_0)\subset\Omega$, for some $R\in (0,1]$. Then, there is a positive constant $N=N(\Omega,R, |\omega|/|B_R|)$ such that
\begin{equation}\label{E: observar sobre medible}
\|\mathcal E_\lambda  f\|_{L^2(\Omega)}\le Ne^{N\sqrt\lambda}\|\mathcal E_\lambda  f\|_{L^1(\omega)},\;\;\mbox{when}\;\;  f\in L^2(\Omega)\;\;\mbox{and}\;\;\lambda>0.
\end{equation}
\end{theorem}

\begin{proof} Without loss of generality we may assume $x_0=0$. Because $B_{4R}\subset \Omega$  and  \eqref{E: desigualdadespectralocalizada} stands,  there is $N=N(\Omega, R)$ such that
\begin{equation}\label{E: desigualdadespectralocalizada2}
\|\mathcal E_\lambda  f\|_{L^2(\Omega)}\le Ne^{N\sqrt\lambda}\|\mathcal E_\lambda  f\|_{L^2(B_R)},\;\;\mbox{when}\;\; f\in L^2(\Omega)\;\;\mbox{and}\;\; \lambda>0.
\end{equation}
For $f\in L^2(\Omega)$ arbitrarily given, define
\begin{equation*}
u(x,y)=\sum_{\lambda_j\le\lambda}( f,e_j)e^{\sqrt{\lambda_j}y}e_j.
\end{equation*}
One can verify that $\Delta u+\partial^2_yu=0$ in $B_{4R}(0,0)\subset\Omega\times\R$. Hence, there are $N=N(n)$ and $\rho=\rho(n)$ such that
\begin{equation*}\label{E: fundamental}
\|\partial^\alpha_x\partial_y^\beta u\|_{L^\infty(B_{2R}(0,0))}\le \frac{N(|\alpha|+\beta)!}{(R\rho)^{|\alpha |+\beta}}\left(\text{\rlap |{$\int_{B_{4R}(0,0)}$}}|u|^2\,dxdy\right)^{\frac 12},\ \text{when}\ \alpha\in\N^n, \beta\ge 1.
\end{equation*}
For the later see \cite[Chapter 5]{Morrey}, \cite[Chapter 3]{FJohn2}. Thus, $\mathcal E_\lambda f$ is a real-analytic function in $B_{2R}$, with the estimates:
\begin{equation*}\label{E: fundamental1}
\|\partial^\alpha_x \mathcal E_\lambda f\|_{L^\infty(B_{2R})}\le N|\alpha|! (R\rho)^{-|\alpha|} \|u\|_{L^\infty(\Omega\times (-4,4))},\ \text{for}\ \alpha\in \N^n.
\end{equation*}
By either extending $|u|$ as zero outside of $\Omega\times\R$, which turns $|u|$ into a subharmonic function in $\R^{n+1}$ or  the local properties of solutions to elliptic equations \cite[Theorems 8.17, 8.25]{GilbargTrudinger} and the orthonormality of $\{e_j: j\ge 1\}$ in $\Omega$, there is $N=N(\Omega)$ such that
\begin{equation*}
\|u\|_{L^\infty(\Omega\times (-4,4))}\le N\|u\|_{L^2(\Omega\times (-5,5))}\le Ne^{N\sqrt\lambda}\|\mathcal E_\lambda f\|_{L^2(\Omega)}.
\end{equation*}
The last two inequalities show that
\begin{equation*}\label{E: fundamental2}
\|\partial^\alpha_x \mathcal E_\lambda f\|_{L^\infty(B_{2R})}\le Ne^{N\sqrt{\lambda}}|\alpha|!\left(R\rho\right)^{-|\alpha |}\|\mathcal E_\lambda f\|_{L^2(\Omega)},\ \text{for}\ \alpha\in\N^n,
\end{equation*}
with $N$ and $\rho$ as above. In particular, $\mathcal E_{\lambda} f$ verifies the hypothesis in Theorem \ref{T:3} with
\begin{equation*}
M=Ne^{N\sqrt{\lambda}}\|\mathcal E_\lambda f\|_{L^2(\Omega)},
\end{equation*}
and there are $N=N(\Omega, R, |\omega|/|B_R|)$ and $\theta=\theta(\Omega, R, |\omega|/|B_R|)$ with
\begin{equation}\label{E: lainterpolacion}
\|\mathcal E_\lambda f\|_{L^\infty(B_R)}\le Ne^{N\sqrt{\lambda}}\|\mathcal E_\lambda f\|_{L^1(\omega)}^\theta\|\mathcal E_\lambda f\|_{L^2(\Omega)}^{1-\theta}.
\end{equation}
Now, the estimate (\ref{E: observar sobre medible})  follows from \eqref{E: desigualdadespectralocalizada2} and \eqref{E: lainterpolacion}.
\end{proof}

\begin{theorem}\label{T: 2global 2-sphere 1-cylinder}
Let $\Omega$, $x_0$, $R$ and $\omega$ be as in Theorem \ref{T: 1spectralinequality}. Then, there are $N=N(\Omega, R, |\omega|/|B_R|)$ and $\theta=\theta(\Omega, R, |\omega|/|B_R|)\in (0,1)$, such that
\begin{equation}\label{E: dos esferas un cilindros desigualdad}
\|e^{t\Delta} f\|_{L^2(\Omega)}\le \left(Ne^{\frac N{t-s}}\|e^{t\Delta} f\|_{L^1(\omega)}\right)^\theta\|e^{s\Delta} f\|_{L^2(\Omega)}^{1-\theta},
\end{equation}
when $0\le s<t$ and $f\in L^2(\Omega)$.
\end{theorem}

\begin{proof}
Let $0\le s<t$ and $f\in L^2(\Omega)$. Since
\begin{equation*}
\|e^{t\Delta}\mathcal E_\lambda^\perp f\|_{L^2(\Omega)}\le e^{-\lambda (t-s)}\|e^{s\Delta}f\|_{L^2(\Omega)},\ \text{when}\ f\in L^2(\Omega),
\end{equation*}
it follows from Theorem \ref{T: 1spectralinequality} that
\begin{equation*}
\begin{split}
\| e^{t\Delta} f\|_{L^2(\Omega)}
&\le \| e^{t \Delta}\mathcal E_\lambda f\|_{L^2(\Omega)}+\| e^{t \Delta}{\mathcal E}_\lambda^\perp f\|_{L^2(\Omega)} \\
&\le Ne^{N\sqrt\lambda}\|e^{t \Delta}\mathcal E_\lambda f\|_{L^1(\omega)}+e^{-\lambda (t-s)}\|e^{s\Delta}f\|_{L^2(\Omega)} \\
&\le Ne^{N\sqrt\lambda}\left[\|e^{t \Delta} f\|_{L^1(\omega)}+\|e^{t \Delta}\mathcal E_\lambda^\perp f\|_{L^2(\omega)}\right]+e^{-\lambda (t-s)}\|e^{s\Delta}f\|_{L^2(\Omega)}\\
&\le Ne^{N\sqrt\lambda}\left[\|e^{t \Delta} f\|_{L^1(\omega)}+e^{-\lambda (t-s)}\|e^{s\Delta}f\|_{L^2(\Omega)}\right].
\end{split}
\end{equation*}
Consequently, it holds that
\begin{equation}\label{E: llegandoalfinal}
\| e^{t \Delta} f\|_{L^2(\Omega)}\le Ne^{N\sqrt\lambda}\left[\|e^{t \Delta} f\|_{L^1(\omega)}+e^{-\lambda (t-s)}\|e^{s\Delta}f\|_{L^2(\Omega)}\right].
\end{equation}
Because
\begin{equation*}\label{E: ladesigualdaddesiempre}
\max_{\lambda\ge 0} e^{A\sqrt\lambda-\lambda (t-s)}\le e^{\frac {N(A)}{t-s}},\ \text{for all}\ A>0,
\end{equation*}
 it follows from \eqref{E: llegandoalfinal}  that for each $\lambda> 0$,
\begin{equation*}\label{E: masdesigualdesdrr}
\| e^{t\Delta} f\|_{L^2(\Omega)}\le\\ Ne^{\frac{N}{t-s}}\left[e^{N\lambda(t-s)}\|e^{t \Delta} f\|_{L^1(\omega)}+e^{-\lambda (t-s)/N}\|e^{s\Delta}f\|_{L^2(\Omega)}\right].
\end{equation*}
Setting $\epsilon= e^{-\lambda(t-s)}$ in the above estimate shows that the inequality
\begin{equation}\label{E: llegandoalfunal}
\| e^{t \Delta} f\|_{L^2(\Omega)}\le Ne^{\frac{N}{t-s}}\left[\epsilon^{-N}\|e^{t \Delta} f\|_{L^1(\omega)}\,+\epsilon\|e^{s\Delta}f\|_{L^2(\Omega)}\right],
\end{equation}
holds, for all $0<\epsilon\le 1$. The minimization of the right hand in \eqref{E: llegandoalfunal} for $\epsilon$ in $(0,1)$, as well as the fact that
\begin{equation*}
\| e^{t\Delta} f\|_{L^2(\Omega)}\le \|e^{s\Delta}f\|_{L^2(\Omega)},\ \text{when}\ t>s,
\end{equation*}
implies Theorem \ref{T: 2global 2-sphere 1-cylinder}.
\end{proof}

\begin{remark}\label{R:primero}
{\it Theorem \ref{T: 2global 2-sphere 1-cylinder} shows that the observability or spectral elliptic inequality \eqref{E: observar sobre medible} implies the inequality \eqref{E: dos esferas un cilindros desigualdad}. In particular, the elliptic spectral inequality \eqref{E: desigualdadespectralocalizada} implies the inequality:
\begin{equation}\label{E: dos esferas un cilindros desigualdad2}
\|e^{t\Delta} f\|_{L^2(\Omega)}\le \left(Ne^{\frac N{t-s}}\|e^{t\Delta} f\|_{L^2(B_R(x_0))}\right)^\theta\|e^{s\Delta} f\|_{L^2(\Omega)}^{1-\theta},
\end{equation}
when $0\le s<t$, $B_{4R}(x_0)\subset\Omega$ and $f\in L^2(\Omega)$. In fact, both \eqref{E: observar sobre medible} and  \eqref{E: dos esferas un cilindros desigualdad} or \eqref{E: desigualdadespectralocalizada} and \eqref{E: dos esferas un cilindros desigualdad2} are equivalent, for if \eqref{E: dos esferas un cilindros desigualdad} holds, take $f=\sum_{\lambda_j\le\lambda}e^{\lambda_j /\sqrt{\lambda}}a_j e_j$, $s=0$ and $t=1/\sqrt{\lambda}$ in \eqref{E: dos esferas un cilindros desigualdad} to derive that
\begin{equation*}\label{E: dos esferas un cilindros desigualdad3}
\Big(\sum_{\lambda_j\le\lambda}a_j^2\Big)^{\frac 12}\le Ne^{N\sqrt\lambda}\Big\| \sum_{\lambda_j\le\lambda}a_j e_j\Big\|_{L^1(\omega)},\ \text{when}\ a_j\in\R,\ j\ge 1, \lambda >0.
\end{equation*}

The interested reader may want here to compare the previous claims, Theorem \ref{T: elteoremauquenuncaquiseescribir} and \cite[Proposition 2.2]{PhungWang1}.}
\end{remark}
\begin{lemma}\label{L: Fubini}  Let  $B_R(x_0)\subset\Omega$ and $\mathcal D\subset B_R(x_0)\times (0,T)$  be a subset of positive measure. Set
\begin{equation*}
\mathcal D_t=\{x\in \Omega : (x,t)\in\mathcal D\},\ E=\{t\in (0,T): |\mathcal D_t |\ge |\mathcal D|/(2T)\},\ t\in (0,T).
\end{equation*}
Then, $\mathcal D_t\subset\Omega$ is measurable for a.e. $t\in (0,T)$, $E$ is measurable in $(0,T)$, $|E|\ge |\mathcal D|/2|B_R|$ and
\begin{equation}\label{E: desigualdadgeometrica}
\chi_E(t)\chi_{\mathcal D_t}(x)\le \chi_{\mathcal D}(x,t),\ \text{in}\ \Omega\times (0,T).
\end{equation}
\end{lemma}

\begin{proof} From Fubini's theorem,
\begin{equation*}
|\mathcal D|=\int_0^T|\mathcal D_t|\,dt=\int_E|\mathcal D_t|\,dt+\int_{[0,T]\setminus E}|\mathcal D_t|\,dt\le|B_R||E|+|\mathcal D|/2.
\end{equation*}
\end{proof}

\begin{theorem}\label{T:4carcajada} Let  $x_0\in \Omega$ and $R\in (0,1]$ be such that $B_{4R}(x_0)\subset\Omega$. Let  $\mathcal D\subset B_R(x_0)\times (0,T)$ be a measurable set with $|\mathcal D|>0$. Write $E$ and $\mathcal D_t$ for the sets associated to $\mathcal D$ in Lemma \ref{L: Fubini}. Then, for each $\eta\in (0,1)$,
   there are $N=N(\Omega,R, |\mathcal D|/\left(T|B_R|\right),\eta)$ and $\theta=\theta(\Omega,R, |\mathcal D|/\left(T|B_R|\right),\eta)$ with $\theta\in (0,1)$, such that
\begin{equation}\label{E: unajotafantastica}
\|e^{t_2\Delta}f\|_{L^2(\Omega)} \le \left( N e^{N/(t_2-t_1)} \int_{t_1}^{t_2}\chi_E(s) \|e^{s\Delta}f\|_{L^1(\mathcal D_s)}\,ds \right)^{\theta}\|e^{t_1\Delta}f\|_{L^2(\Omega)}^{1-\theta},
\end{equation}
when $0\le t_1<t_2\le T$, $|E\cap (t_1,t_2)|\ge \eta (t_2-t_1)$ and $f\in L^2(\Omega)$. Moreover,
\begin{equation}\label{E: hayqueverqiebiensalestovayajuergamaravi}
\begin{split}
&e^{-\frac{N+1-\theta}{t_2-t_1}}\|e^{t_2\Delta}f\|_{L^2(\Omega)}- e^{-\frac{N+1-\theta}{q\left(t_2-t_1\right)}}\|e^{t_1\Delta}f\|_{L^2(\Omega)}\\
&\le N\int_{t_1}^{t_2}\chi_E(s) \|e^{s\Delta}f\|_{L^1(\mathcal D_s)}\,ds,\;\;\mbox{when}\;\;q\ge (N+1-\theta)/(N+1).
\end{split}
\end{equation}
\end{theorem}

\begin{proof}  After removing from $E$ a set with zero Lebesgue measure, we may assume that $\mathcal D_t$ is measurable for all $t$ in $E$. From Lemma \ref{L: Fubini}, $\mathcal D_t\subset B_R(x_0)$, $B_{4R}(x_0)\subset\Omega$ and $|\mathcal D_t|/|B_R|\ge |\mathcal D|/(2T|B_R|)$, when $t$ is in $E$. From Theorem \ref{T: 2global 2-sphere 1-cylinder}, there are $N=N(\Omega,R, |\mathcal D|/\left(T|B_R|\right))$ and $\theta=\theta(\Omega,R, |\mathcal D|/\left(T|B_R|\right))$ such that
\begin{equation}\label{E: globalssphere1cylinder}
\|e^{t\Delta} f\|_{L^2(\Omega)}\le \left(Ne^{\frac N{t-s}}\|e^{t\Delta} f\|_{L^1(\mathcal D_t)}\right)^\theta\|e^{s\Delta} f\|_{L^2(\Omega)}^{1-\theta},
\end{equation}
when $0\le s<t$, $t\in E$ and $f\in L^2(\Omega)$. Let $\eta\in (0,1)$ and $0\le t_1<t_2\le T$ satisfy $|E\cap (t_1,t_2)|\ge \eta (t_2-t_1)$. Set $\tau=t_1+\frac\eta 2\,\left(t_2-t_1\right)$. Then
\begin{equation}\label{E: necesarisimo}
|E\cap (\tau,t_2)|=|E\cap (t_1,t_2)|-|E\cap (t_1,\tau)|\ge \frac\eta 2 (t_2-t_1).
\end{equation}
From  \eqref{E: globalssphere1cylinder} with $s=t_1$ and the decay property of $\|e^{t\Delta}f\|_{L^2(\Omega)}$, we get
\begin{equation}\label{E: 2globalssphere1cylinder}
\|e^{t_2\Delta} f\|_{L^2(\Omega)}\le \left(Ne^{\frac N{t_2-t_1}}\|e^{t\Delta} f\|_{L^1(\mathcal D_t)}\right)^\theta\|e^{t_1\Delta} f\|_{L^2(\Omega)}^{1-\theta},\ t\in E\cap (\tau,t_2).
\end{equation}
 The inequality \eqref{E: unajotafantastica} follows from  the integration with respect to $t$ of \eqref{E: 2globalssphere1cylinder} over $E\cap (\tau,t_2)$,  H\"older's inequality with $p=1/\theta$ and \eqref{E: necesarisimo}.

The inequality \eqref{E: unajotafantastica} and Young's inequality imply that
\begin{equation}\label{E: quebienfunciona2}
\begin{split}
&\|e^{t_2\Delta}f\|_{L^2(\Omega)} \le \\
&\e \|e^{t_1\Delta}f\|_{L^2(\Omega)} +\e^{-\frac{1-\theta}\theta} N e^{\frac{N}{t_2-t_1}} \int_{t_1}^{t_2}\chi_E(s) \|e^{s\Delta}f\|_{L^1(\mathcal D_s)}\,ds,\;\;\mbox{when}\;\;\e>0.
\end{split}
\end{equation}
 Multiplying first \eqref{E: quebienfunciona2} by $\e^{\frac{1-\theta}\theta} e^{-\frac{N}{t_2-t_1}}$ and then replacing $\e$ by $\e^\theta$, we get that
\begin{equation*}
\begin{split}\label{E: hayqueverqiebiensalesto}
&\e^{1-\theta} e^{-\frac{N}{\left(t_2-t_1\right)}} \|e^{t_2\Delta}f\|_{L^2(\Omega)}- \e\, e^{-\frac{N}{\left(t_2-t_1\right)}}\|e^{t_1\Delta}f\|_{L^2(\Omega)}\\
&\le N\int_{t_1}^{t_2}\chi_E(s) \|e^{s\Delta}f\|_{L^1(\mathcal D_s)}\,ds,\;\;\mbox{when}\;\;\e>0.
\end{split}
\end{equation*}
 Choosing  $\e=e^{-\frac 1{t_2-t_1}}$ in the above inequality leads to
 \begin{equation*}\label{E: hayqueverqiebiensalestovayajuerga}
 \begin{split}
&e^{-\frac{N+1-\theta}{\left(t_2-t_1\right)}} \|e^{t_2\Delta}f\|_{L^2(\Omega)}- e^{-\frac{N+1}{\left(t_2-t_1\right)}}\|e^{t_1\Delta}f\|_{L^2(\Omega)}\\
&\le N\int_{t_1}^{t_2}\chi_E(s) \|e^{s\Delta}f\|_{L^1(\mathcal D_s)}\,ds.
\end{split}
\end{equation*}
This implies \eqref{E: hayqueverqiebiensalestovayajuergamaravi}, for $q\ge \frac{N+1-\theta}{N+1}$.
\end{proof}

The reader can find the proof of the following  Lemma \ref{L: 2teoriamedida} in either \cite[pp. 256-257]{JLLions} or  \cite[Proposition 2.1]{PhungWang1}.
\begin{lemma}\label{L: 2teoriamedida} Let $E$ be a subset of positive measure in $(0,T)$. Let  $l$ be  a density point of E. Then, for each $z > 1$, there is  $l_1=l_1(z,E)$  in $(l, T)$ such that, the sequence $\{l_{m}\}$ defined as
\begin{equation*}
l_{m+1}= l+z^{-m}\left(l_1-l\right),\ m=1, 2,\cdots,
\end{equation*}
verifies
\begin{equation}\label{E: intervalosbienencajados}
|E\cap (l_{m+1}, l_m)|\ge \frac 13 \left(l_m-l_{m+1}\right),\ \text{when}\ m\ge 1.
\end{equation}
\end{lemma}
\vskip 10pt

 \begin{proof} [Proof of Theorem~\ref{T:5carcajada}] Let $E$ and $\mathcal D_t$ be the sets associated to $\mathcal D$ in Lemma \ref{L: Fubini} and $l$ be a density point in $E$. For $z>1$ to be fixed later, $\{l_m\}$ denotes the sequence associated to $l$ and $z$ in Lemma \ref{L: 2teoriamedida}. Because \eqref{E: intervalosbienencajados} holds,
we may apply Theorem \ref{T:4carcajada},  with $\eta=1/3$, $t_1=l_{m+1}$ and $t_2=l_{m}$, for  each $m\ge 1$, to get that there are $N=N(\Omega,R, |\mathcal D|/\left(T|B_R|\right))>0$  and $\theta=\theta(\Omega,R, |\mathcal D|/\left(T|B_R|\right))$, with $\theta\in (0,1)$, such that
\begin{equation}\label{E: hayqueverqiebiensalestovayajuergamaravicasi}
\begin{split}
&e^{-\frac{N+1-\theta}{l_m-l_{m+1}}}\|e^{l_m\Delta}f\|_{L^2(\Omega)}- e^{-\frac{N+1-\theta}{q\left(l_m-l_{m+1}\right)}}\|e^{l_{m+1}\Delta}f\|_{L^2(\Omega)}\\
&\le N\int_{l_{m+1}}^{l_{m}}\chi_E(s) \|e^{s\Delta}f\|_{L^1(\mathcal D_s)}\,ds,\;\;\mbox{when}\;\;q\ge \frac{N+1-\theta}{N+1}\;\;\mbox{and}\;\;m\ge 1.
\end{split}
\end{equation}
 Setting $z=1/q$ in \eqref{E: hayqueverqiebiensalestovayajuergamaravicasi} (which leads to $1<z\le \frac{N+1}{N+1-\theta}$) and
\begin{equation*}
\gamma_z(t)=e^{-\frac{N+1-\theta}{\left(z-1\right)\left(l_1-l\right)t}},\ t>0,
\end{equation*} recalling that
\begin{equation*}
l_m-l_{m+1}=z^{-m}\left(z-1\right)\left(l_1-l\right),\ \text{for}\ m\ge 1,
\end{equation*}
 we have
\begin{equation}\label{E: hayqueverqiebiensalestovayajuergamaravi22222}
\begin{split}
\gamma_z(z^{-m})\|e^{l_m\Delta}f\|_{L^2(\Omega)}- \gamma_z(z^{-m-1})\|e^{l_{m+1}\Delta}f\|_{L^2(\Omega)}\\
\le N\int_{l_{m+1}}^{l_{m}}\chi_E(s) \|e^{s\Delta}f\|_{L^1(\mathcal D_s)}\,ds,\;\;\mbox{when}\;\;m\ge 1.
\end{split}
\end{equation}
 Choose now
\begin{equation*}
z=\frac 12\left(1+\frac{N+1}{N+1-\theta}\right).
\end{equation*}
The choice of $z$ and Lemma \ref{L: 2teoriamedida} determines  $l_1$ in $(l,T)$ and from \eqref{E: hayqueverqiebiensalestovayajuergamaravi22222},
\begin{equation}\label{E: hayqueverqiebiensalestovayajuergamaravi222224}
\begin{split}
\gamma(z^{-m})\|e^{l_m\Delta}f\|_{L^2(\Omega)}- \gamma(z^{-m-1})\|e^{l_{m+1}\Delta}f\|_{L^2(\Omega)}\\
\le N\int_{l_{m+1}}^{l_{m}}\chi_E(s) \|e^{s\Delta}f\|_{L^1(\mathcal D_s)}\,ds,\;\;\mbox{when}\;\;m\ge 1.
\end{split}
\end{equation}
 with
\begin{equation*}\label{E: otraformulaincreiblequealucianal}
\gamma(t)=e^{-A/t}\ \text{and}\ A=A(\Omega,R, E, |\mathcal D|/\left(T|B_R| \right))=\frac{2\left(N+1-\theta\right)^2}{\theta\left(l_1-l\right)}\, .
\end{equation*}
Finally, because of
\begin{equation*}
\|e^{T\Delta}f\|_{L^2(\Omega)}\le \|e^{l_1\Delta}f\|_{L^2(\Omega)},\  \sup_{t\ge 0}\|e^{t\Delta}f\|_{L^2(\Omega)}<+\infty,\ \lim_{t\to 0+}\gamma(t)=0,
\end{equation*}
and \eqref{E: desigualdadgeometrica}, the addition of the telescoping series in  \eqref{E: hayqueverqiebiensalestovayajuergamaravi222224} gives
\begin{equation*}
\|e^{T\Delta}f\|_{L^2(\Omega)}\le Ne^{zA}\int_{\mathcal D\cap(\Omega\times [l,l_1])}|e^{t\Delta}f(x)|\,dxdt,\;\;\mbox{for}\;\; f\in L^2(\Omega),
\end{equation*}
which proves \eqref{E: observabilidadparabolica} with $B=zA+\log N$.

\end{proof}

\begin{remark}\label{R: 3}
{\it The constant $B$ in Theorem \ref{T:5carcajada} depends on $E$ because the choice of $l_1=l_1(z,E)$ in Lemma \ref{L: 2teoriamedida} depends on the possible complex structure of the measurable set $E$ (See the proof of Lemma  \ref{L: 2teoriamedida} in \cite[Proposition 2.1]{PhungWang1}).  When $\mathcal D=\omega\times (0,T)$, one may take $l=T/2$, $l_1=T$, $z=2$ and then,
\begin{equation*}
B=A(\Omega, R, |\omega|/|B_R|)/T.
\end{equation*}}
\end{remark}

\begin{remark}\label{R: unaremarjakoquejode}
{\it The proof of Theorem \ref{T:5carcajada} also implies the following observability estimate:
\begin{equation*}\label{E: otraacotacionhhh}
\sup_{m\ge 0}\sup_{l_{m+1}\le t\le l_m}e^{-z^{m+1}A}\|e^{t\Delta}f\|_{L^2(\Omega)}\le N\int_{\mathcal D\cap (\Omega\times [l,l_1])}|e^{t\Delta}f(x)|\,dxdt,
\end{equation*}
for $f$ in $L^2(\Omega)$, and with $z$, $N$ and $A$ as defined along the proof of Theorem \ref{T:5carcajada}. Here, $l_0=T$.}
\end{remark}

\section{Spectral inequalities}\label{S:6}

Throughout this section, $\Omega$ is a bounded domain in $\mathbb{R}^n$ and $\nu_q$ is the unit exterior normal vector associated to $q\in\p\om$.

\definition\label{D: Lipschitz domain}  $\om$
is a Lipschitz domain (sometimes called strongly Lipschitz or Lipschitz graph domains) with constants $m$ and $\varrho$ when for each point $p$ on the boundary of $\Omega$ there is a rectangular coordinate system $x=(x',x_n)$ and a Lipschitz function $\phi:\R^{n-1}\longrightarrow \R$ verifying
\begin{equation}\label{E: condicionLipschitz}
\phi(0')=0,\quad |\phi(x_1')-\phi(x_2')|\le m|x_1'-x_2'|,\ \text{for all}\ x_1', x_2'\in\R^{n-1},
\end{equation}
$p=(0',0)$ on this coordinate system  and
\begin{equation}\label{E: segunacondicionlipschitz}
\begin{split}
&Z_{m,\varrho}\cap\Omega=\{(x',x_n) : |x'|< \varrho,\ \phi(x')< x_n < 2m\varrho \},\\
&Z_{m,\varrho}\cap\partial\Omega= \{(x',\phi(x')) : |x'|< \varrho\},
\end{split}
\end{equation}
where $Z_{m,\varrho}=B'_	\varrho\times (-2m\varrho,2m\varrho)$.

\definition\label{D: C1} $\Omega$ is a $C^1$ domain when it is a Lipschitz domain and the functions $\varphi$ associated to points $p$ in $\partial\Omega$ satisfying \eqref{E: condicionLipschitz} and \eqref{E: segunacondicionlipschitz} are in $C^1(\R^{n-1})$. Under such condition, there is
\begin{equation*}
\theta: [0,+\infty)\longrightarrow [0,+\infty),\ \text{with}\ \theta\ \text{nondecreasing},\ \lim_{t\to 0^+}\theta(t)=0,
\end{equation*}
with
\begin{equation}\label{E: algoquesecumpleenc1}
|\left(q-p\right)\cdot \nu_q|\le |p-q|\theta(|p-q|),\ \text{when}\ p, q\in\partial\Omega.
\end{equation}

\definition\label{D: lowerC1} $\Omega$ is a lower $C^1$ domain when it is a Lipschitz domain and
\begin{equation}\label{E: unamalditacondici—nparapoligocosm}
\liminf_{q\in\partial\Omega,\, q\to p}\frac{\left(q-p\right)\cdot\nu_q}{|q-p|}\ge 0,\ \text{for each}\ p\in\partial\Omega.
\end{equation}

\begin{remark}\label{R: algobastanteimportantene} {\it  Lipschitz polygons in the plane, convex domains in $\Rn$ (See \cite[p. 72, Lemma 3.4.1]{Morrey} for a proof that convex domains in $\Rn$ are Lipschitz domains (or strongly Lipschitz domains,  to keep pace with Morrey's definition)) or Lipschitz polyhedron in $\Rn$, $n\ge 3$, are lower $C^1$ domains. In fact, in all these cases and for $p\in\partial\Omega$, there is $r_p>0$ such that
\begin{equation*}
\left(q-p\right)\cdot\nu_q\ge 0,\ \text{for a.e.}\ q\in B_{r_p}(p)\cap\partial\Omega.
\end{equation*}

When $\Omega$ is convex, $r_p=+\infty$. In general, a Lipschitz domain $\Omega$ is a lower $C^1$ domain when the Lipschitz functions $\phi$ describing its boundary can be discomposed as the sum of two Lipschitz functions, $\phi=\phi_1+\phi_2$, with $\phi_1$ convex over $\R^{n-1}$ and $\phi_2$ satisfying
\begin{equation*}
\lim_{\e\to 0^+}\sup_{\ \ \ |x_1'-x_2'|<\e}|\nabla\phi_2(x_1')-\nabla\phi_2(x_2')|=0.
\end{equation*}
In particular, $C^1$ domains are lower $C^1$ domains (See \eqref{E: algoquesecumpleenc1}).}
\end{remark}
\definition\label{D: contractable} A Lipschitz domain $\Omega$ in $\Rn$ is locally star-shaped when for each $p\in\partial\Omega$ there are $x_p$ in $\Omega$ and $r_p>0$ such that
\begin{equation}\label{E: definiciocontratcet}
|p-x_p|<r_p\ \text{and}\ B_{r_p}(x_p)\cap\Omega\ \text{is star-shaped with center}\ x_p.
\end{equation}
In particular,
\begin{equation*}
\left(q-x_p\right)\cdot\nu_q\ge0,\ \text{for a.e.}\  q \in B_{r_p}(x_p)\cap\partial\Omega.
\end{equation*}
\begin{remark}\label{R: algpautilizarsinduda}
{\it The compactness of $\partial\Omega$ shows that when $\Omega$ is locally star-shaped, there are a finite set $\mathcal A\subset \Omega$, $0<\e, \rho\le 1$ and a family of positive numbers $0<r_x\le 1$, $x\in\mathcal A$, such that
\begin{equation}\label{E:lasmaravillasqueocurren}
\begin{split}
\partial\Omega\subset &\cup_{x\in\mathcal A} B_{r_x}(x),\  B_{\left(1+\e\right)r_x}(x)\cap\Omega\ \text{is}\ \text{star-shaped},\\
&\mathcal A\subset\Omega^{4\rho}\ \text{and}\ \overline\Omega\setminus\Omega^{4\rho}\subset \cup_{x\in\mathcal A} B_{r_x}(x).
\end{split}
\end{equation}
Here,
\begin{equation*}
\Omega^\eta=\{x\in\Omega : d(x,\partial\Omega)>\eta\},\ \text{when}\ \eta>0.
\end{equation*}}
\end{remark}
\begin{remark}\label{R: algoasombroso}
{\it Theorem \ref{T: unapropiedadgemotetrica} below shows that Lispchitz polygons in the plane, $C^1$ domains, convex domains and Lipschitz polyhedron in $\R^n$, $n\ge 3$ are  locally star-shaped. Lipschitz domains in $\Rn$ with Lipschitz constant $m<\frac 12$ are also locally star-shaped. Recall that not all the polygons in the plane or polyhedra in $\R^3$ are Lipschitz domains (See \cite[p. 496, pp. 508-509]{VerchotaVogel} or the two-brick domain of \cite[p. 303]{KralWendland}).}
\end{remark}

\begin{theorem}\label{T: unapropiedadgemotetrica} Let $\Omega$ be a  lower $C^1$ domain. Then, $\Omega$ is locally star-shaped.
\end{theorem}
\begin{proof} Let $(x',x_n)$ and $\phi$ be accordingly the rectangular coordinate system and Lipschitz function associated to $p\in\partial\Omega$, satisfying  \eqref{E: condicionLipschitz} and \eqref{E: segunacondicionlipschitz}. Let then, $x_p=p+\delta\, e$, $e=(0',1)$, and $r_p=2\delta$,
where $\delta>0$ will be chosen later.
Clearly $x_p$ is in $\Omega$, $|p-x_p|< r_p$ and $B_{r_p}(x_p)\subset Z_{m,\varrho}$, when $0<\delta<\min{\{\tfrac \varrho2, \tfrac{2m\varrho}3\}}$. Moreover, for almost every $q$ in $B_{r_p}(x_p)\cap\partial\Omega$, it holds that $q=(x',\phi(x'))$ for some $x'$ in $B'_\rho$,
\begin{equation*}
\nu_q=\frac{\left(\nabla\phi(x'),-1\right)}{\sqrt{1+|\nabla\phi(x')|^2}}\quad \text{and}\quad |q-p|\le r_p+\delta=3\delta
\end{equation*}
and
\begin{equation*}
\begin{split}
\left(q-x_p\right)\cdot\nu_q&=\left(q-p\right)\cdot\nu_q-\delta\, e\cdot\nu_q\\
&\ge \left(q-p\right)\cdot\nu_q+\frac{\delta}{\sqrt{1+m^2}}.
\end{split}
\end{equation*}
From \eqref{E: unamalditacondici—nparapoligocosm} there is $s_p>0$ such that
\begin{equation}\label{E: consecunciadeexistenciade}
\left(q-p\right)\cdot\nu_q\ge -\frac{|q-p|}{3\sqrt{1+m^2}},\ \text{when}\ q\in\partial\Omega\ \text{and}\   |q-p|<s_p.
\end{equation}
Thus, \eqref{E: definiciocontratcet} holds for the choices we made of $x_p$, $r_p$ and $s_p$, provided that $\delta$ is chosen with
\begin{equation*}
0<\delta<\min{\{\tfrac \varrho2, \tfrac{2m\varrho}3, \tfrac{s_p}3\}}.
\end{equation*}
\end{proof}

The proof of Theorem \ref{T: elteoremauquenuncaquiseescribir} will follow from the following Lemmas \ref{L: convexidadlogaritmica} and \ref{L: interpolacionenlafrontera}.

\begin{lemma}\label{L: convexidadlogaritmica} Let $\Omega$  be a Lipschitz domain in $\Rn$, $R>0$ and assume that $B_R(x_0)\cap\Omega$ is star-shaped with center at some $x_0\in\overline\Omega$. Then,
 \begin{equation}\label{E: laconvexidadnecesaria}
 \|u\|_{L^2(B_{r_2}(x_0)\cap\Omega)}\le \|u\|_{L^2(B_{r_1}(x_0)\cap\Omega)}^{\theta}\|u\|_{L^2(B_{r_3}(x_0)\cap\Omega)}^{1-\theta},\ \text{with}\ \theta=\frac{\log{\frac{r_3}{r_2}}}{\log{\frac{r_3}{r_1}}}\ ,
\end{equation}
when $\Delta u=0$ in $B_R(x_0)\cap\Omega$, $u=0$ on $\triangle_R(x_0)$ and $0<r_1<r_2<r_3\le R$.
\end{lemma}

\begin{remark}\label{R: queeslorelevantedeserstar} {\it See \cite[Lemma 3.1]{KukavicaNystrom}  for a proof of Lemma \ref{L: convexidadlogaritmica} and \cite{adolfssonEscauriazakenig} and \cite{adolfssonEscauriaza} for related results but with spheres replacing balls. The relevance of the assumption that $B_R(x_0)\cap\Omega$ is star-shaped in the proof of \cite[Lemma 3.1]{KukavicaNystrom} is that
\begin{equation*}
 \left(q-x_0\right)\cdot\nu_q\ge 0,\ \text{for a.e.}\  q\ \text{in}\  B_R(x_0)\cap\partial\Omega
 \end{equation*}
 and certain terms arising in the arguments can be drop because of their nonnegative sign. In fact, \eqref{E: laconvexidadnecesaria} is the logarithmic convexity of the $L^2$-norm of $u$ over $B_r(x_0)\cap\Omega$ for $0<r\le R$ with respect to the variable $\log r$. Lemma \ref{L: convexidadlogaritmica}  extends up to the boundary the classical interior three-spheres inequality for harmonic functions, first stablished for complex analytic functions by Hadamard \cite{Hadamard1} and extended for harmonic functions by several authors \cite{GarofaloLin}:
\begin{equation}\label{E: laconvexidadnecesariainterior}
 \|u\|_{L^2(B_{r_2}(x_0))}\le \|u\|_{L^2(B_{r_1}(x_0))}^{\theta}\|u\|_{L^2(B_{r_3}(x_0))}^{1-\theta},\ \text{with}\ \theta=\frac{\log{\frac{r_3}{r_2}}}{\log{\frac{r_3}{r_1}}}\ ,
\end{equation}
\emph{when $\Delta u=0$ in $\Omega$, $B_R(x_0)\subset\Omega$ and $0<r_1<r_2<r_3\le R$.}
\vspace{0.3cm}

When $B_R(x_0)\subset\Omega$, $B_R(x_0)\cap\Omega$ is star-shaped, and a proof for \eqref{E: laconvexidadnecesariainterior} is within the one for  \cite[Lemma 3.1]{KukavicaNystrom}.}
\end{remark}

\begin{lemma}\label{L: interpolacionenlafrontera}
Let $\Omega$ be a  Lipischitz domain with constants $m$ and $\varrho$, $0<r<\varrho/10$ and $q$ be in $\partial\Omega$. Then, there are $N=N(n,m)$ and $\theta=\theta(n,m)$, $0<\theta<1$, such that the inequality
\begin{equation*}
\|u\|_{L^2(B_r(q)\cap\Omega)}\le Nr^{\frac{3\theta}2}\|\tfrac{\partial u}{\partial\nu}\|_{L^2(\Delta_{6r}(q))}^\theta\|u\|_{L^2(B_{8r}(q)\cap\Omega)}^{1-\theta},
\end{equation*}
holds for all harmonic functions $u$ in $B_{8r}(q)\cap\Omega$ satisfying $u=0$ on $B_{8r}(q)\cap\partial\Omega$.
\end{lemma}
To prove Lemma \ref{L: interpolacionenlafrontera} we use the Carleman inequality in Lemma \ref{L: carkemanb‡sico}. As far as the authors know, the first $L^2$-type Carleman inequality with a radial weight and whose proof was worked out in Cartesian coordinates appeared first in \cite[Lemma 1]{kenigwang}. Lemma \ref{L: carkemanb‡sico} borrows ideas from \cite[Lemma 1]{kenigwang} but the proof here is somehow simpler.  In \cite[p. 518]{adolfssonEscauriazakenig} also appears an interpolation inequality similar to the one in Lemma \ref{L: interpolacionenlafrontera} but with the $L^2$-norms replaced by $L^1$-norms. The inequality in \cite[p. 518]{adolfssonEscauriazakenig} holds though its proof in \cite{adolfssonEscauriazakenig} is not correct. It follows from Lemma \ref{L: interpolacionenlafrontera} and properties of harmonic functions.
\begin{lemma}\label{L: carkemanb‡sico} Let $\Omega$ be a Lipschitz domain in $\Rn$ with $\Omega\subset B_R$ and $\tau>0$. Assume that $0$ is not in $\overline\Omega$. Then,
\begin{equation*}
\int_{\Omega}|x|^{-2\tau}u^2\,dx\le \frac{R^2}{4\tau^2}\int_{\Omega}|x|^{-2\tau+2}\left(\Delta u\right)^2\,dx
-\frac{R^2}{2\tau}\int_{\partial\Omega}q\cdot\nu\,|q|^{-2\tau}\left(\tfrac{\partial u}{\partial\nu}\right)^2\,d\s,
\end{equation*}
for all $u$ in $C^2(\overline \Omega)$ satisfying $u=0$ on $\partial\Omega$.
\end{lemma}
\begin{proof}[Proof of Lemma \ref{L: carkemanb‡sico}]
Let $u$ be in $C^2(\overline \Omega)$, $u=0$ on $\partial\Omega$ and define $f=|x|^{-\tau}u$. Then,
\begin{equation}\label{E:70}
|x|^{1-\tau } \Delta u =|x|\left(\Delta f+\tfrac{\tau ^{2}}{|x|^{2}}\,f\right)+ \tfrac{2\tau}{|x|}\,\left(x\cdot \nabla f+\tfrac{n-2}{2}f\right).
\end{equation}
Square both sides of \eqref{E:70} to get
\begin{multline}\label{E:800}
|x|^{2-2\tau}\left(
\Delta u \right)^{2}= 4\tau^2|x|^{-2}\,\left(x\cdot \nabla f+\tfrac{n-2}{2}f\right)^{2}+|x|^2\left(\triangle f+\tfrac{\tau^2}{|x|^{2}}\,f\right)^2\\+4\tau\left(x\cdot \nabla f+\tfrac{n-2}{2}f\right)\left(\Delta
f+\tfrac{\tau^{2}}{|x|^{2}}f\right).
\end{multline}
Observe that
\begin{equation*}
\int_{\Omega} |x|^{-2}\left(x\cdot \nabla f+\tfrac{n-2}{2}f\right)f\,dx=\tfrac 12\int_{\partial\Omega}\tfrac{q\cdot \nu}{|q|^2}f^2\,d\s=0,
\end{equation*}
and integrate \eqref{E:800} over $\Omega$. Then we get the identity
\begin{multline}\label{E:80}
\int_{\Omega}|x|^{2-2\tau}\left(
\Delta u\right)^{2}\,dx=
4\tau^2\int_{\Omega} |x|^{-2}\left(x\cdot \nabla f+\tfrac{n-2}{2}f\right)^{2}\,dx\\+\int_{\Omega} |x|^{2}\left(\triangle f+\tfrac{\tau^2}{|x|^{2}}\,f\right)^2\,dx
+4\tau\int_{\Omega}\left(x\cdot \nabla f+\tfrac{n-2}{2}f\right)\Delta f\,dx.
\end{multline}
The Rellich-Ne\v cas or Pohozaev identity
\begin{equation*}
 \nabla\cdot\left[x |\nabla f|^{2}-2\left(x\cdot
\nabla f\right)\nabla f\right]+2\left(x \cdot \nabla f\right)\Delta f=\left(n-2\right)|\nabla f|
^{2}\
\end{equation*}
and the identity
\begin{equation*}
\Delta \left( f^{2}\right) =2f\Delta f+2\left| \nabla f\right| ^{2},
\end{equation*}
give the formula
\begin{equation*}
 \nabla\cdot\left[x|\nabla f|^2-2\left(x\cdot\nabla
f\right)\nabla f\right]+2\left(x\cdot \nabla f+\tfrac{n-2}{2}f\right)\Delta f=
\tfrac{n-2}{2}\Delta(f^2).
\end{equation*}
The integration of
this identity over $\Omega$ implies the formula
\begin{equation}\label{E:130}
4\int_{\Omega}\left(x\cdot \nabla f+\tfrac{n-2}{2}f\right)\Delta f\,dx=
2\int_{\partial\Omega}q\cdot\nu\,\left(\tfrac{\partial f}{\partial\nu}\right)^2d\s,
\end{equation}
and plugging
\eqref{E:130} into
\eqref{E:80} gives the identity
\begin{multline}\label{E: joquebienfunciona}
\int_{\Omega}|x|^{2-2\tau}\left(
\Delta u\right)^{2}\,dx=
4\tau^2\int_{\Omega} |x|^{-2}\left(x\cdot \nabla f+\tfrac{n-2}{2}f\right)^{2}\,dx\\+\int_{\Omega} |x|^{2}\left(\triangle f+\tfrac{\tau^2}{|x|^{2}}\,f\right)^2\,dx
+2\tau\int_{\partial\Omega}q\cdot\nu\,\left(\tfrac{\partial f}{\partial\nu}\right)^2d\s.
\end{multline}
Next, the identity
\begin{equation*}
-\int_{\Omega} \left(x\cdot \nabla f+\tfrac{n-2}{2}f\right)f\,dx=\int_{\Omega}f^2\,dx,
\end{equation*}
the Cauchy-Schwarz's inequality and $\Omega\subset B_R$ show that
\begin{equation}\label{E: vayacasualidadquetirfunciona}
R^{-2}\int_{\Omega} f^2\,dx\le \int_{\Omega}|x|^{-2}\left(x\cdot \nabla f+\tfrac{n-2}{2}f\right)^2\,dx.
\end{equation}
Then, Lemma \ref{L: carkemanb‡sico} follows from \eqref{E: joquebienfunciona} and \eqref{E: vayacasualidadquetirfunciona}.
\end{proof}

\vskip 10 pt
\begin{proof}[Proof of Lemma \ref{L: interpolacionenlafrontera}] After rescaling and translation we may assume that $q=0$, $r=1$ and that there is a Lipschitz function $\phi:\R^{n-1}\longrightarrow \R$ verifying
 \begin{equation}\label{E: segunacondicionlipschitzescalada}
\begin{split}
&\phi(0')=0,\quad |\phi(x_1')-\phi(x_2')|\le m|x_1'-x_2'|,\ \text{for all}\ x_1', x_2'\in\R^{n-1},\\
&Z_{m,10}\cap\Omega=\{(x',x_n) : |x'|< 10,\ \phi(x')< x_n < 20m \},\\
&Z_{m,10}\cap\partial\Omega= \{(x',\phi(x')) : |x'|< 10\},
\end{split}
\end{equation}
where $Z_{m,10}=B'_{10}\times (-20m,20m)$. Recalling that $e=(0',1)$. Thus, $-e\notin\overline\Omega$
and
\begin{equation}\label{E: unacadenadecontenidos}
B_1\subset B_2(-e)\subset B_5(-e)\subset B_6.
\end{equation}
Choose $\psi$ in $C_0^\infty(B_5(-e))$ with $\psi\equiv 1$ in $B_3(-e)$ and $\psi\equiv0$ outside $B_4(-e)$. From \eqref{E: segunacondicionlipschitzescalada}, there is $\beta\in(0,1)$, $\beta=\beta(m)$,
such that
\begin{equation}\label{E: loLipschitzianocosigiuemaravi}
B_{2\beta}(-e)\cap \Omega=\emptyset.
\end{equation}
By translating the
inequality in Lemma \ref{L: carkemanb‡sico} from $0$ to $-e$  we find that
\begin{equation}\label{E: ellLemma8 deformamasclasica}
\begin{split}
&\left\| |x+e|^{-\tau}f\right\|_{L^2(B_5(-e)\cap\Omega)}\le \left\||x+e|^{-\tau+1}\Delta f\right\|_{L^2(B_5(-e)\cap\Omega)}\\
&+\left\||q+e|^{\frac 12-\tau}\tfrac{\partial f}{\partial\nu}\right\|_{L^2(B_5(-e)\cap\partial\Omega)},
\end{split}
\end{equation}
when $f$ is  in $C^2_0(B_5(-e)\cap\overline \Omega)$, $f=0$ on $B_5(-e)\cap\partial\Omega$ and $\tau\ge 20$.
Let then $u$ be harmonic in $B_8\cap \Omega$ with
$u=0$ on $B_8\cap\partial\Omega$ and
take $f=u\psi$ in \eqref{E: ellLemma8 deformamasclasica}. We get
\begin{equation}\label{E: hayqueverquededesigualda}
\begin{split}
&\left\| |x+e|^{-\tau}u\psi\right\|_{L^2(B_5(-e)\cap\Omega)}
\le \left\| |x+e|^{1-\tau}\left(u\Delta\psi+2\nabla\psi\cdot\nabla u\right) \right\|_{L^2(B_5(-e)\cap\Omega)}\\
&+\left\||q+e|^{\frac 12-\tau}\psi
\tfrac{\p u}{\p \nu}\right\|_{L^2(B_5(-e)\cap\partial\Omega)},\;\;\mbox{when}\;\;\tau\ge 20.
\end{split}
\end{equation}
 Next, because $\psi\equiv1$ on $B_3(-e)$ and \eqref{E: unacadenadecontenidos}, we have
\begin{equation}\label{E: laprimera}
2^{-\tau}\|u\|_{L^2(B_1\cap \Omega)}\le\left\| |x+e|^{-\tau}f\right\|_{L^2(B_5(-e)\cap\Omega)},\ \text{when}\ \tau\ge 1.
\end{equation}
Also, it follows from from \eqref{E: loLipschitzianocosigiuemaravi} that
\begin{equation}\label{E; lasegundaquenece}
\left\||q+e|^{\frac 12-\tau}\psi
\tfrac{\p u}{\p \nu}\right\|_{L^2(B_5(-e)\cap\partial\Omega)}\le \beta^{-\tau}\|\tfrac{\p u}{\p \nu}\|_{L^2(\Delta_6)},\ \text{when}\ \tau\ge 1\ .
\end{equation}
Now, $|\Delta\psi|+|\nabla\psi|$ is supported in $B_4(-e)\setminus B_3(-e)$, where $|x+e|\ge 3$ and
\begin{equation}\label{E: laterceradsigual}
\begin{split}
&\left\||x+e|^{1-\tau}\left(u\Delta\psi+2\nabla\psi\cdot\nabla u\right) \right\|_{L^2(B_5(-e)\cap\Omega)}\\
&\le N3^{-\tau}\big\||u|+|\nabla u|\big\|_{L^2(B_4(-e)\setminus B_3(-e)\cap\Omega)}\\
&\le  N3^{-\tau}\|u\|_{L^2(B_6\cap\Omega)},\;\;\mbox{when}\;\;\tau\ge 1,\;\;\mbox{with}\;\;N=N(n).
\end{split}
\end{equation}
 Putting together \eqref{E: hayqueverquededesigualda}, \eqref{E: laprimera}, \eqref{E; lasegundaquenece} and \eqref{E: laterceradsigual}, we get
\begin{equation*}
\|u\|_{L^2(B_1\cap \Omega)}\le N\left[\left(\tfrac 23\right)^{\tau}\|u\|_{L^2(B_6\cap\Omega)}+\left(\tfrac 2\beta\right)^{\tau}\|\tfrac{\p u}{\p \nu}\|_{L^2(\Delta_6)}\right],\ \text{when}\ \tau\ge 20.
\end{equation*}
The later inequality shows that there is $N=N(n,m)\ge 1$ such that
\begin{equation}\label{E: algoaminimizar}
\|u\|_{L^2(B_1\cap \Omega)}\le N\left[\e\|u\|_{L^2(B_8\cap\Omega)}+\e^{-N}\|\tfrac{\p u}{\p \nu}\|_{L^2(\Delta_6)}\right],\ \text{when}\ 0<\e\le \left(\tfrac 23\right)^{20}.
\end{equation}
Finally, the minimization of \eqref{E: algoaminimizar} shows that Lemma \ref{L: interpolacionenlafrontera} holds for $q=0$ and $r=1$ and completes the proof.
\end{proof}

\vskip 10 pt

\begin{proof}[Proof of Theorem \ref{T: elteoremauquenuncaquiseescribir}] Without loss of generality we may assume that $x_0=0$ and $B_{4R}\subset\Omega$ for some  fixed $0<R\le 1$.  To prove that \eqref{E: desigualdadespectralocalizada} holds with $x_0=0$,  we first show that under the hypothesis of Theorem \ref{T: elteoremauquenuncaquiseescribir}, there are $\rho$ and $\theta$ in $(0,1)$ and $N\ge 1$ such that the inequality
\begin{equation}\label{E: laprieradesigualdaimportante}
\|u\|_{L^2(\Omega\times[ -1,1])}\le N\|u\|_{L^2(\Omega^{2\rho}\times[ -2,2])}^\theta\|u\|_{L^2(\Omega\times[ -4,4])}^{1-\theta}
\end{equation}
holds when
\begin{equation}\label{E: la ecuaciondesiempre}
\begin{cases}
\Delta u+\partial^2_yu=0,\ &\text{in}\ \Omega\times\R,\\
u=0,\ &\text{on}\ \partial\Omega\times\R.
\end{cases}
\end{equation}
To prove \eqref{E: laprieradesigualdaimportante}, we  recall that \eqref{E:lasmaravillasqueocurren} holds and $\mathcal A=\{x_1,x_2,\dots, x_l\}$ for some $x_i$ in $\Omega^{4\rho}$, $i=1, \dots , l$, for some $l\ge 1$. Because $B_{\left(1+\e\right)r_{x_i}}(x_i)\cap\Omega$ is star-shaped with center $x_i$,  the same holds for $B_{\left(1+\e\right)r_{x_i}}(x_i,\tau)\cap\Omega\times\R$ and with center $(x_i,\tau)$, $i=1,\dots,l$, for all $\tau\in\R$. Then, from Lemma \ref{L: convexidadlogaritmica}, it follows that
 \begin{equation}\label{E: laconvexidadnecesariaque hacefaltaahroa}
\|u\|_{L^2(B_{r_{x_i}}(x_i,\tau)\cap\Omega\times\R)}\le \|u\|_{L^2(B_{\rho}(x_i,\tau)\cap\Omega\times\R)}^{\theta}\|u\|_{L^2(B_{\left(1+\e\right)r_{x_i}}(x_i,\tau)\cap\Omega\times\R)}^{1-\theta},
\end{equation}
for $i=1,\dots,l$ with
\begin{equation}
 \theta=\min{\left\{\log{\left(1+\e\right)}/\log{\left(\frac{\left(1+\e\right)r_{x_i}}{\rho}\right)}: i=1,\dots,l\right\}}.
 \end{equation}
 Also, from \eqref{E: laconvexidadnecesariainterior}
 \begin{equation}\label{E: laconvexidadnecesariaque hacefaltaahroatambiensegunda}
 \|u\|_{L^2(B_{\rho}(x,\tau))}\le \|u\|_{L^2(B_{\frac \rho 4}(x,\tau))}^{\theta}\|u\|_{L^2(B_{2\rho}(x,\tau))}^{1-\theta},\ \text{with}\ \theta=\frac{\log{2}}{\log{8}}
 \end{equation}
when $x$ is in the closure of $\Omega^{4\rho}$ and $\tau\in\R$. Because \eqref{E:lasmaravillasqueocurren}, \eqref{E: laconvexidadnecesariaque hacefaltaahroa}, \eqref{E: laconvexidadnecesariaque hacefaltaahroatambiensegunda} hold and there are $z_1,\dots,z_m$ in the closure of $\Omega^{4\rho}$ with
\begin{equation*}
\overline{\Omega^{4\rho}}\subset\cup_{i=1}^m B_{\rho}(z_i),
\end{equation*}
it is now clear that there are $N=N(l,m,\rho,\mathcal A)\ge 1$ and a new $\theta=\theta(l,m,\rho,\mathcal A)$ in $(0,1)$ such that \eqref{E: laprieradesigualdaimportante} holds.

 From \eqref{E: laconvexidadnecesariainterior} with $r_1=\frac r4$, $r_2=r$, $r_3=2r$, we see that
 \begin{equation}\label{E: Hadamard1}
\|u\|_{L^2(B_{r}(x_1,y_1))}\le \|u\|_{L^2(B_{\frac r4}(x_1,y_1))}^{\theta}\|u\|_{L^2(B_{2r}(x_1,y_1))}^{1-\theta},\ \theta=\frac{\log{2}}{\log{8}},
\end{equation}
when $B_{2r}(x_1,y_1)\subset\Omega\times [-4,4]$. From \eqref{E: Hadamard1} and the fact that
 \begin{equation*}
 B_{\frac r4}(x_1,y_1)\subset B_r(x_2,y_2),\ \text{when}\  (x_2,y_2)\in B_{\frac r4}(x_1,y_1),
 \end{equation*}
we find that
\begin{equation}\label{E: Hadamard2}
\|u\|_{L^2(B_{r}(x_1,y_1))}\le \|u\|_{L^2(B_{r}(x_2,y_2))}^{\theta}\|u\|_{L^2(\Omega\times [-4,4])}^{1-\theta},\ \theta=\frac{\log{2}}{\log{8}},
\end{equation}
when $B_{2r}(x_1,y_1)\subset\om\times [-4,4]$, $(x_2,y_2)$ is in $B_{\frac r4}(x_1,y_1)$ and $r>0$.
Because $\rho$ is now fixed and $\overline\Omega$ is compact, there are $0<r\le 1$ and $k\ge 2$, which depend on $\rho$, the geometry of $\Omega$ and $R$, such that for all $(x_0,y_0)$ in $\Omega^{2\rho}\times [-2,2]$ there are $k$ points $(x_1,y_1), (x_2,y_2),\dots,(x_k,y_k)$ in $\Omega^{2\rho}\times [-2,2]$ with
\begin{equation*}
\begin{split}
B_{2r}(x_i,y_i)\subset\om &\times [-4,4],\ (x_{i+1},y_{i+1})\in B_{\frac r4}(x_i,y_i),\ \text{for}\  i=0,\dots,k-1,\\
&(x_k,y_k)=(0,\tfrac R{16})\ \text{and}\ B_r(0,\tfrac R{16})\subset B_{\frac R8}^+(0,0).
\end{split}
\end{equation*}
The later and \eqref{E: Hadamard2} show that
\begin{equation}\label{E: Hadamard3}
\|u\|_{L^2(B_{r}(x_i,y_i))}\le \|u\|_{L^2(B_{r}(x_{i+1},y_{i+1}))}^{\theta}\|u\|_{L^2(\Omega\times [-4,4])}^{1-\theta},\ \text{for}\ i=0,\dots,k-1,
\end{equation}
while the compactness of $\Omega$ and the iteration of \eqref{E: Hadamard3} imply that for some new $N= N(l,m,k,\rho,\mathcal A, R)\ge 1$ and $0<\theta=\theta(l,m,k,\rho,\mathcal A, R)<1$
\begin{equation}\label{E: Hadamard 4}
\|u\|_{L^2(\Omega^{2\rho}\times [-2,2])}\le N\|u\|_{L^2( B_{\frac R8}^+(0,0))}^{\theta} \|u\|_{L^2(\Omega\times [-4,4])}^{1-\theta}.
\end{equation}
Putting  together \eqref{E: laprieradesigualdaimportante} and \eqref{E: Hadamard 4}, it follows that when $u$ satisfies \eqref{E: la ecuaciondesiempre},
\begin{equation}\label{E: Hadamard 5}
\|u\|_{L^2(\Omega\times [-1,1])}\le N\|u\|_{L^2( B_{\frac R8}^+(0,0))}^{\theta} \|u\|_{L^2(\Omega\times [-4,4])}^{1-\theta},
\end{equation}
with $N$ and $\theta$ as before. Finally, proceeding as in \cite{G. LebeauL. Robbiano}, for $\lambda>0$ we take
\begin{equation}\label{E: unabuenaelecion}
u(x,y)=\sum_{\lambda_j\le\lambda}a_j\,\tfrac{\sinh{\left(\sqrt{\lambda_j}\,y\right)}}{\sqrt{\lambda_j}}\,e_j(x),
\end{equation}
with $\{a_j\}$ a sequence of real numbers. It satisfies \eqref{E: la ecuaciondesiempre} and also $u(x,0)\equiv 0$ in $\overline\Omega$. Then, we may assume that $\tfrac R8\le \tfrac{\varrho}{10}$ and Lemma \ref{L: interpolacionenlafrontera} implies
\begin{equation}\label{E:usando ladesigualdadeinterpola}
\|u\|_{L^2(B_{\frac R8}(0,0)^+)}\le N\|\tfrac{\partial u}{\partial y}(\cdot,0)\|_{L^2(B_{R})}^\theta\|u\|_{L^2(\Omega\times [-4,4])}^{1-\theta}\, .
\end{equation}
Combining \eqref{E: Hadamard 5} with \eqref{E:usando ladesigualdadeinterpola} leads to
\begin{equation*}\label{E: Hadamard 6}
\|u\|_{L^2(\Omega\times [-1,1])}\le N\|\tfrac{\partial u}{\partial y}(\cdot,0)\|_{L^2(B_{R})}^\theta \|u\|_{L^2(\Omega\times [-4,4])}^{1-\theta}.
\end{equation*}
 This, along with \eqref{E: unabuenaelecion} and standard arguments in \cite{G. LebeauE. Zuazua}, shows that $\Omega$ satisfies the condition \eqref{E: desigualdadespectralocalizada}.
\end{proof}
\begin{remark}\label{R: otraanotacioncojonuda} {\it The reader can now easily derive with arguments based on the Lemmas \ref{L: convexidadlogaritmica}, \ref{L: interpolacionenlafrontera} and Theorem \ref{T: elteoremauquenuncaquiseescribir}, the following spectral inequality:
\begin{theorem}\label{T: otradesigualdaddeobservacionfrontera}
Suppose that $\Omega$ is Lipschitz and verifies \eqref{E: desigualdadespectralocalizada}.
Let $q_0\in\partial\Omega$, $\tau\in\R$ and $R\in(0, 1]$.  Then, there is $N=N(\Omega, R,\tau)$ such that the inequalities
\begin{equation*}\label{E queagitamientotengoyadeesto}
\left(\sum_{\lambda_j\le \lambda} a_j^2+b_j^2\right)^{\frac 12}\le e^{N\sqrt{\lambda}}\left\|\sum_{\lambda_j\le\lambda}(a_j e^{\sqrt{\lambda_j}\,y}+b_je^{-\sqrt{\lambda_j}\,y})\tfrac{\partial e_j}{\partial\nu}\right\|_{L^2(B_R(q_0,\tau)\cap\partial\Omega\times\R)},
\end{equation*}
hold for all sequences $\{a_j\}$ and $\{b_j\}$ in $\R$. In particular, the above inequality holds when $\Omega$ is a  bounded  Lipschitz and locally star-shaped domain.
\end{theorem}}
\end{remark}
\section{Boundary observability}\label{S:5}
 Throughout this section $\Omega$ is a bounded Lipschitz domain in $\mathbb{R}^n$ and $T$ is a positive time.
We first study quantitative estimates of real analyticity with respect to the space-time variables for caloric functions in $\Omega\times (0,T)$ with zero lateral Dirichlet boundary conditions.
Let
\begin{equation}\label{E: definiticiofunciondegreen}
G(x,y,t)=\sum_{j\ge1}e^{-\lambda_jt}e_j(x)e_j(y)
\end{equation}
be the Green's function for $\Delta-\partial_t$ on $\Omega$ with zero lateral Dirichlet boundary condition. By the maximum principle
\begin{equation*}
0\le G(x,y,t)\le\left(4\pi t\right)^{-\frac n2}e^{-\frac{|x-y|^2}{4t}},
\end{equation*}
and
\begin{equation}\label{E: algo que ayuda}
G(x,x,t)=\sum_{j\ge 1}e^{-\lambda_jt}e_j(x)^2\le \left(4\pi t\right)^{-\frac n2}.
\end{equation}
 \begin{definition}\label{D: fromteralocalrealanalitica} Let $q_0\in\partial\Omega$ and $0<R\le 1$. We say that $\triangle_{4R}(q_0)$ is real-analytic with constants $\varrho$ and $\delta$ if
for each $q\in \triangle_{4R}(q_0)$, there are a new rectangular coordinate system where $q=0$, and a real-analytic function $\phi: B'_{\varrho}\subset\mathbb{R}^{n-1}
\rightarrow \mathbb{R}$ verifying
\begin{equation*}
\begin{split}
&\phi(0')=0,\;\;|\partial^{\alpha}\phi(x')|\leq |\alpha|!\delta^{-|\alpha|-1},\;\;\mbox{when}\;\;x'\in B'_{\varrho},\;\alpha\in\mathbb{N}^{n-1},\\
&B_\varrho\cap\Omega=B_\varrho\cap\{(x',x_n):x'\in B'_{\varrho},\;\; x_n>\phi(x')\},\\
&B_\varrho\cap\partial\Omega=B_\varrho\cap\{(x',x_n):x'\in B'_{\varrho},\;\; x_n=\phi(x')\}.
\end{split}
\end{equation*}
\end{definition}
\noindent Here, $B'_\varrho$ denotes the open ball of radius $\varrho$ and with center at $0'$ in $\mathbb{R}^{n-1}$.

\begin{lemma}\label{L: analiticidadcalorica} Let  $q_0\in\partial\Omega$ and  $R\in (0, 1]$. Assume that $\triangle_{4R}(q_0)$ is real-analytic with constants $\varrho$ and $\delta$. Then, there are $N=N(\varrho,\delta)$, $\rho=\rho(\varrho,\delta)$, with $0<\rho\le1$, such that
\begin{equation}\label{E: elfinal de estamalditoterme44}
|\partial_x^\alpha\partial_t^\beta e^{t\Delta}f(x)|\le\frac{N\left(t-s\right)^{-\frac n4}\, e^{8R^2/\left(t-s\right)} |\alpha|!\,\beta!}{(R\rho)^{|\alpha |}\left(\left(t-s\right)/4\right)^\beta}\|e^{s\Delta}f\|_{L^2(\Omega)},
\end{equation}
when $x\in B_{2R}(q_0)\cap\overline\Omega$, $0\le s<t$, $\alpha\in \N^n$ and $\beta\ge 0$.
\end{lemma}
\begin{proof} It suffices to prove (\ref{E: elfinal de estamalditoterme44}) for the case when $s=0$. Let $f=\sum_{ j\geq 1}a_je_j$.  Set then
\begin{equation*}
\begin{aligned}
&u(x,t,y)=\sum_{j=1}^{+\infty}a_j e^{-\lambda_j t+\sqrt{\lambda_j}y}e_j(x),\;\ x\in\overline\Omega, t>0, y\in\R, \\
&u(x,t)=u(x,t,0)=\sum_{j=1}^{+\infty} a_j e^{-\lambda_jt} e_j(x)= e^{t\Delta}f(x),x\in\overline\Omega, t>0.
\end{aligned}
\end{equation*}
 We have
\begin{equation*}
\begin{cases}
\partial^2_y\partial_t^\beta u+\Delta \partial_t^\beta u=0,\ &\text{in}\ \Omega\times\R,\\
\partial_t^\beta u=0,\ &\text{on}\  \partial\Omega\times\R,
\end{cases}
\end{equation*}
\begin{equation}\label{E: la formula de la derivada}
\partial_t^\beta u(x,t,y)=\left(-1\right)^\beta\sum_{j\ge 1}a_j\lambda_j^\beta e^{-\lambda_jt+\sqrt{\lambda_j}y}e_j(x).
\end{equation}
Because $\triangle_{4R}(q_0)$ is real-analytic, there are $N=N(\varrho, \delta)$ and $\rho=\rho(\varrho, \delta)$ such that
\begin{equation}\label{E: fundamental3}
\begin{split}
&\|\partial^\alpha_x\partial_t^\beta u(\cdot,t,\cdot)\|_{L^\infty(B_{2R}(q_0,0)\cap\Omega\times\R)}\\ &\le \frac{N |\alpha|!}{(R\rho)^{|\alpha |}} \left(\text{\rlap |{$\int_{B_{4R}(q_0,0)\cap\Omega\times\R}$}}|\partial_t^\beta u(x,t,y)|^2\,dxdy\right)^{\frac 12},\;\;\mbox{when}\;\;\alpha\in\N^n, \beta\ge 0.
\end{split}
\end{equation}
 For the later see \cite[Chapter 5]{Morrey} and \cite[Chapter 3]{FJohn2}. Now, it follows from \eqref{E: la formula de la derivada} that
\begin{equation*}\label{E: meenccantaestaformula}
|\partial_t^\beta u(x,t,y)|\le \Big(\sum_{j\ge 1}a_j^2\Big)^{\frac 12}\Big(\sum_{j\ge 1}\lambda_j^{2\beta}e^{-2\lambda_jt+2\sqrt{\lambda_j}y}e_j(x)^2\Big)^{\frac12}.
\end{equation*}
Also, it stands that
\begin{multline*}\label{E; unabuena secuencia de formulas}
\sum_{j\ge 1}\lambda_j^{2\beta}e^{-2\lambda_jt+2\sqrt{\lambda_j}y}e_j(x)^2=t^{-2\beta}e^{y^2/t}\sum_{j\ge 1}\left(\lambda_jt\right)^{2\beta}e^{-\lambda_jt}e_j(x)^2e^{-\left(\sqrt{\lambda_jt}-\frac y{\sqrt t}\right)^2}\\
\le t^{-2\beta}e^{y^2/t}\sum_{j\ge 1}\left(\lambda_jt\right)^{2\beta}e^{-\lambda_jt}e_j(x)^2=t^{-2\beta}e^{y^2/t}\sum_{j\ge 1}\left(\lambda_jt\right)^{2\beta}e^{-\lambda_jt/2}e^{-\lambda_jt/2}e_j(x)^2.
\end{multline*}
 Next Stirling's formula shows that
\begin{equation*}\label{E: hayque ver que fornulazas mesalen}
\max_{x\ge 0}x^{2\beta}e^{-x/2}=4^{4\beta}\beta^{2\beta}e^{-2\beta}\lesssim 4^{4\beta}\left(\beta!\right)^2,\;\;\mbox{when}\;\; \beta\ge 0.
\end{equation*}
Finally, the above three inequalities, as well as  \eqref{E: algo que ayuda},  imply that
\begin{equation}\label{E: totalmente necesariom}
|\partial_t^\beta u(x,t,y)|\le \Big(2\pi t\Big)^{-\frac n4}\, e^{y^2/2t}\Big(\sum_{j\ge 1}a_j^2\Big)^{\frac 12}\left(t/4\right)^{-\beta}\beta!,
\end{equation}
for $t>0$, $x\in\overline\Omega$, $y\in\R$ and $\beta\ge 0$. The later inequality and \eqref{E: fundamental3} imply that
\begin{equation*}\label{E: elfinal de estamalditoterme}
|\partial_x^\alpha\partial_t^\beta u(x,t)|\le\frac{Nt^{-\frac n4}\, e^{8R^2/t} |\alpha|!\,\beta!}{(R\rho)^{|\alpha |}\left(t/4\right)^\beta}\Big(\sum_{j\ge 1}a_j^2\Big)^{\frac 12},
\end{equation*}
when $x\in B_{2R}(q_0)\cap \overline{\Omega}$, $t>0$, $\alpha\in \N^n$ and $\beta\ge 0$.
\end{proof}

The next caloric interpolation inequality plays the same role for the boundary case as the inequality (\ref{E: dos esferas un cilindros desigualdad2}) for the interior case.

\begin{theorem}\label{interpolation}
Let $\Omega$ be a bounded  Lipschitz domain in $\mathbb{R}^n$ with
constants $(m,\varrho)$ and satisfy the condition (\ref{E: desigualdadespectralocalizada}). Then, given $0<R\leq1$, $0\leq t_1< t_2\leq T\leq1$
and $q\in\partial\Omega$, there are $N=N(\Omega,R,m,\varrho)\geq1$
and $\theta=\theta(m,\varrho)$, with $\theta\in(0,1)$, such that
\begin{equation*}
\|e^{t_2\Delta}f\|_{L^2(\Omega)}\leq \left(Ne^{\frac{N}{t_2-t_1}}\|\tfrac{\partial}{\partial\nu}\,e^{t\Delta}f\|_{L^2(\Delta_{R}(q)\times
(t_1,t_2))}\right)^\theta\|e^{t_1\Delta}f\|^{1-\theta}_{L^2(\Omega)},
\;f\in L^2(\Omega).
\end{equation*}
\end{theorem}
To prove  Theorem \ref{interpolation}, we need first some
lemmas. We  begin with the following Carleman inequality (See \cite{EscauriazaFernandez1} and \cite{EscauriazaSereginSverak}).
\begin{lemma}\label{lemma1}
Let $\Omega$ be a bounded Lipschitz domain in $\mathbb{R}^n$ with
constants $(m,\varrho)$, $0\notin\Omega$ and $\sigma(t)=te^{-Mt}$, $M>0$. Then
\begin{equation*}
\begin{split}
&\sqrt{\tau M}\|\sigma(t)^{-\tau}e^{-|x|^2/8t}h\|_{
L^2(\Omega\times(0,+\infty))}\\&
\leq \|t^{1/2}\sigma(t)^{-\tau}e^{-|x|^2/8t}(\Delta+\partial_{t})h\|
_{L^2(\Omega\times(0,+\infty))}\\
&+\||q|^{1/2}\sigma(t)^{-\tau}e^{-|q|^{2}/8t}\tfrac{\partial h}{\partial\nu}
\|_{{
L^2(\partial\Omega\times(0,+\infty))}},
\end{split}
\end{equation*}
when $\tau\geq1$ and $h\in C_{0}^\infty(\overline{\Omega}\times[0,+\infty))$ with
$h=0$ on $\partial\Omega\times[0,+\infty)$.
\end{lemma}
\begin{proof}
First, let $f=\sigma(t)^{-\tau}e^{-|x^2|/8t}h$. Then
\begin{equation*}
\begin{split}
&\sigma(t)^{-\tau}e^{-|x|^2/8t}(\Delta+\partial_t)h=
\sigma(t)^{-\tau}e^{-|x|^2/8t}(\Delta+\partial_t)
(\sigma(t)^{\tau}e^{|x|^2/8t}f)\\
&=\Delta  f-\tfrac{|x|^2}{16t^2}\,f+\tau\partial_t(\log\sigma)
f+\partial_tf+\tfrac{x}{2t}\cdot\nabla f+\tfrac{n}{4t}\,f\ .
\end{split}
\end{equation*}
Thus
\begin{equation}\label{zhang1}
\begin{split}
&\|t^{1/2}\sigma(t)^{-\tau}e^{-|x|^2/8t}(\Delta+\partial_t)h\|
_{L^2(\Omega\times(0,+\infty))}^2\\
&=\|t^{1/2}(\Delta f-\tfrac{|x|^2}{16t^2}f+\tau\partial_t(\log\sigma)f)
\|_{L^2(\Omega\times(0,+\infty))}^2\\
&+\|t^{1/2}(\partial_t f+\tfrac{x}{2t}\cdot\nabla f+\tfrac{n}{4t}f)\|
_{L^2(\Omega\times(0,+\infty))}^2\\
&+\int_{\Omega\times(0,+\infty)}2t
(\Delta f-\tfrac{|x|^2}{16t^2}\,f+\tau\partial_t(\log\sigma)f)
(\partial_tf+\tfrac{x}{2t}\cdot\nabla f+\tfrac{n}{4t}\,f)\,dxdt.
\end{split}
\end{equation}

Next, integrating by parts  we have the following two identities:
\begin{equation}\label{E: primeraformula}
\begin{split}
&\int_{\Omega\times(0,+\infty)}2t\partial_tf\Delta f\,dxdt\\
&=\int_{\Omega\times(0,+\infty)}-t\partial_t|\nabla f|^2+
2t\nabla\cdot(\partial_tf\nabla f)\,dxdt
=\int_{\Omega\times(0,+\infty)}|\nabla f|^2\,dxdt.
\end{split}
\end{equation}
\begin{equation}\label{E: segundaformual}
\begin{split}
&\int_{\Omega\times(0,+\infty)}2t\partial_tf\left(-\tfrac{|x|^2}
{16t^2}\,f
+\tau\partial_t(\log\sigma)f\right)\,dxdt\\
&=\int_{\Omega\times(0,+\infty)}-\tfrac{|x|^2}{16t}\,\partial_tf^2
+\tau t\partial_t(\log\sigma)\partial_tf^2\,dxdt\\
&=\int_{\Omega\times(0,+\infty)}-\tfrac{|x|^2}{16t^2}\,f^2-\tau\partial
_t\left(t\partial_t(\log\sigma)\right)f^2\,dxdt.
\end{split}
\end{equation}
The Rellich-Ne\v cas or Pohosaev identity
\begin{equation*}
\begin{split}
&\nabla\cdot\left[\left(x|\nabla f|^2\right)-2\left(x\cdot\nabla f\right)\nabla f\right]=
(n-2)|\nabla f|^2-2\left(x\cdot\nabla f\right)\Delta f\\
&=(n-2)|\nabla f|^2-2\left(x\cdot\nabla f+\tfrac{n}{2}\,f\right)\Delta f+nf\Delta f\\
&=\tfrac{n}{2}\,\Delta(f^2)-2|\nabla f|^2-2\left(x\cdot\nabla f+\tfrac{n}{2}\,f\right)\Delta f
\end{split}
\end{equation*}
gives the formula
\begin{equation*}
\left(x\cdot\nabla f+\tfrac{n}{2}\,f\right)\Delta f=
\nabla\cdot\left[\left(x\cdot\nabla f\right)\nabla f-\tfrac x2\,|\nabla f|^2\right]+\tfrac{n}{4}\,\Delta (f^2)-|\nabla f|^2.
\end{equation*}
Integrating the above identity in $\Omega$,
we get that for each $t>0$,
\begin{equation}\label{E: terceraformula}
\int_{\Omega}\left(x\cdot\nabla f+\tfrac{n}{4}\,f\right)\Delta f\,dx=\tfrac{1}{2}\int_{\partial\Omega}q\cdot\nu\left(
\tfrac{\partial f}{\partial\nu}\right)^2\,d\sigma
-\int_{\Omega}|\nabla f|^2\,dx.
\end{equation}

On the other hand,
\begin{equation}\label{E: cuartaformula}
\begin{split}
&\int_{\Omega\times(0,+\infty)}\left(x\cdot\nabla f+\tfrac{n}{2}\,f\right)\left(-\tfrac{|x|^2}{16t^2}\,f+\tau\partial_t
(\log\sigma)f\right)\,dxdt\\
&=\int_{\Omega\times(0,+\infty)}\tfrac{x}{2}\cdot\nabla\left(f^2\right)\left
(-\tfrac{|x|^2}{16t^2}+\tau\partial_t\left(\log\sigma\right)\right)\,dxdt\\
&+\tfrac{n}{2}\int_{\Omega\times(0,+\infty)}
\left(\tau\partial_t\left(\log\sigma\right)-\tfrac{|x|^2}{16t^2}\right)f^2\,dxdt\\
&=\int_{\Omega\times(0,+\infty)}\tfrac{|x|^2}{16t^2}\,f^2\,dxdt.
\end{split}
\end{equation}
Combining (\ref{zhang1}) and the identities \eqref{E: primeraformula}, \eqref{E: segundaformual}, \eqref{E: terceraformula} and \eqref{E: cuartaformula}, we get
\begin{equation*}
\begin{split}
&\|t^{1/2}\sigma(t)^{-\tau}e^{-|x|^2/8t}(\Delta+\partial_t)h
\|_{L^2(\Omega\times(0,+\infty))}^2\\
&\geq -\tau\int_{\Omega\times(0,+\infty)}\partial_t\left(t\partial_t\left
(\log\sigma\right)\right)f^2\,dxdt+\int_{\partial\Omega\times(0,+\infty)}
\tfrac{1}{2}
q\cdot\nu\left(\tfrac{\partial f}{\partial\nu}\right)^2\,d\sigma dt.
\end{split}
\end{equation*}
Choose then, $\sigma(t)=te^{-Mt},M>0$, $\partial_t(t\partial_t
(\log\sigma))=-M$, and it leads to the desired estimate.
\end{proof}
In Lemmas \ref{lemma2}, \ref{lemma3}, \ref{lemma4} and \ref{lemma6}, we assume that $\Omega$ is a Lipschitz domain with constants $(m,\varrho)$ and $0\in\partial\Omega$. In Lemmas \ref{lemma2} and \ref{lemma6} we also assume that $\Omega$ is near the origin, the region above the graph, $x_n=\phi(x')$, with $\phi$ as in  \eqref{E: condicionLipschitz} and \eqref{E: segunacondicionlipschitz}, so that $-\rho\, e$ is not in $\Omega$, when $0<\rho\le m\varrho$, $e=(0',1)$.
\begin{lemma}\label{lemma2} Let $\sigma(t)=te^{-t}$. Then, for $0<\rho\le m\varrho $,
\begin{equation*}
\begin{split}
&\sqrt{\tau}\|\sigma(t)^{-\tau}e^{-|x+\rho e|^2/8t}h\|_{
L^2(\Omega\times(0,+\infty))}\\
&\leq \|t^{1/2}\sigma(t)^{-\tau}e^{-|x+\rho e|^2/8t}\left(\Delta+\partial_{t}\right)h\|_{L^2(\Omega\times(0,+\infty))}\\
&+\| |q+\rho e|^{1/2}\sigma(t)^{-\tau}
e^{-|q+\rho e|^{2}/8t}\tfrac{\partial h}{\partial\nu}\|_{L^2(\partial\Omega\times(0,+\infty))},
\end{split}
\end{equation*}
when $\tau\geq1$ and $h\in C_{0}^\infty(\overline{\Omega}\times[0,+\infty))$ with
$h=0$ on $\partial\Omega\times[0,+\infty)$.
\end{lemma}
\begin{proof}
It follows from Lemma~\ref{lemma1}
by translation.
\end{proof}
\begin{lemma}\label{lemma3}
There is $N=N(n)$ such that
\begin{equation*}
\|\nabla u\|_{L^2(B_R\cap\Omega\times(0,T/2))}
\leq N\sqrt{R^{-2}+T^{-1}}\,\|u\|_{L^2(B_{4R/3}\cap
\Omega\times(0,T))},
\end{equation*}
when $u$ satisfies $\partial_tu+\Delta u=0$ in $B_{4R/3}\cap\Omega\times[0,T]$, $u=0$ on $\triangle_{4R/3}\times[0,T]$, for some $R,\, T>0$.
\end{lemma}
\begin{proof}
Let $\psi\in C_0^\infty(\mathbb{R}^n)$ and $\alpha\in C^\infty(\mathbb{R})$ be such that $\psi=1$ in $B_R$, $\psi=0$ outside $B_{7R/6}$,$0\leq\psi\leq 1$, $|\Delta \psi|\leq NR^{-2}$;
$\alpha=1$ in $(-\infty,T/2]$, $\alpha=0$ in $[\tfrac{3T}4,+\infty)$, $0\leq\alpha\leq1$ and
$|\partial_t\alpha|\leq N/T$. Then
\begin{equation*}
\begin{split}
&\int_{B_R\cap\Omega\times(0,T/2))}2|\nabla u|^2\,dxdt
\leq \int_{\Omega\times(0,+\infty)}\psi(x)\alpha(t)\left(\Delta+\partial_t\right)(u^2)
\,dxdt\\
&=\int_{\Omega\times(0,+\infty)}(\Delta-\partial_t)
(\psi(x)\alpha(t))
u^2(x,t)\,dxdt-\int_{\Omega}\psi(x)u^2(x,0)\,dx\\
&\leq N(R^{-2}+T^{-1})\int_{B_{4R/3}\cap\Omega\times(0,T)}u^2\,dxdt.
\end{split}
\end{equation*}
\end{proof}
\begin{lemma}\label{lemma4}
Let $0<\rho\leq 1$ and $T>0$. Then, there is $N=N(n)$
such that
\begin{equation}\label{E: desigualdad de energilocal}
\|u(t)\|_{L^2(B_{3\rho}\cap\Omega)}\geq \tfrac{1}{2}
\|u(0)\|_{L^2(B_\rho\cap\Omega)},
\end{equation}
when
\begin{equation}\label{E: condicionquecumplet}
0<t\leq \min\left\{\tfrac{T}{2},\,\tfrac{\rho^2}{N\log^{+}\left(
\tfrac{N\|u\|_{L^2(B_4\cap\Omega\times[0,T])}}
{\rho\|u(0)\|_{L^2(B_\rho\cap\Omega)}}
\right)}\right\}
\end{equation}
and $u$ verifies $\partial_tu+\Delta u=0$ in $B_{4}\cap\Omega\times[0,T]$, $
u=0$ on $\triangle_4\times[0,T]$.
\end{lemma}
\noindent Here,  $\log^{+}x=\max\{\log x,0\}$ and $\frac{1}{0}=+\infty$.
\begin{proof}
Let $\psi\in C_0^\infty(B_{3\rho})$ verify $\psi=1$ in $B_{2\rho}$, $0\leq\psi\leq 1$ and $\rho|\nabla\psi|+\rho^2|D^2\psi|\leq N$. Set $f(x,t)=u(x,t)\psi(x)$. Then
\begin{equation*}
|\Delta f+\partial_tf|\leq N\left(\rho^{-2}|u|+\rho^{-1}|\nabla u|\right)\chi_{_{B_{3\rho}\setminus B_{2\rho}}},\;\;\text{in}\;\;
\Omega\times[0,T].
\end{equation*}
Define
\begin{equation*}
H(t)=\int_{\Omega}f^2(x,t)G(x-y,t)\,dx,\;\;\text{when}\;\;t>0,\;\;
y\in
B_{\rho}\cap\Omega,
\end{equation*}
with $G(x,t)=(4\pi t)^{-\frac{n}{2}}e^{-\frac{|x|^2}{4t}}$. We have
\begin{equation*}
\begin{split}
&\tfrac{d}{dt}\,H(t)=2\int_{\Omega}\left(f(\Delta f+\partial_tf)+|\nabla f|^2\right)G(x-y,t)\,dx\\
&\geq -N\int_{(B_{3\rho}\setminus B_{2\rho})\cap\Omega}
\left(\rho^{-2}|u|^2+|\nabla u|^2\right)G(x-y,t)\,dx\\
&\geq -Nt^{-n/2}e^{-\rho^2/4t}\int_{B_{3\rho}\cap\Omega}
\rho^{-2}|u|^2+|\nabla u|^2\,dx\\
&\geq -N\rho^{-n}e^{-\rho^2/8t}\int_{B_{3\rho}\cap\Omega}
\rho^{-2}|u|^2+|\nabla u|^2\,dx,
\end{split}
\end{equation*}
for $0<t<T$. Integrating the above inequality in $(0,t)$,
\begin{equation*}
\int_{\Omega}f^2(x,t)G(x-y,t)\,dx-u^2(y,0)
\geq -N\rho^{-n}e^{-\rho^2/8t}\!\!\!\int_{B_{3\rho}\cap\Omega}
\int_{0}^{t}\rho^{-2}|u|^2+|\nabla u|^2\,d\tau dx
\end{equation*}
This, along with Lemma \ref{lemma3}
with $R=3\rho$, $\frac{T}{2}=t$, shows that
\begin{equation*}
\begin{split}
&\int_{B_{3\rho}\cap\Omega}u^2(x,t)G(x-y,t)\,dx-u^2(y,0)\\
&\geq -N\rho^{-n-2}(1+\rho^2/t)e^{-\rho^2/8t}
\int_{B_{4\rho}\cap\Omega}\int_0^{2t}u^2(x,\tau)\,dxd\tau\\
&\geq -N\rho^{-n-2}e^{-\rho^2/16t}
\int_{B_{4\rho}\cap\Omega}\int_0^{2t}u^2(x,\tau)\,dxd\tau,
\end{split}
\end{equation*}
when $y\in B_{\rho}\cap\Omega$, $0<2t\leq T$. Integrating
the above inequality for $y\in B_{\rho}\cap\Omega$ and
recalling that
\begin{equation*}
\int_{\mathbb{R}^n}G(x-y,t)\,dy=1,\;\;\text{for all}\;\;x\in
\mathbb{R}^n,\ t>0,
\end{equation*}
we get
\begin{equation*}
\int_{B_{3\rho}\cap\Omega}u^2(x,t)\,dx\geq
\int_{B_{\rho}\cap\Omega}u^2(y,0)\,dy
-N\rho^{-2}e^{-\rho^2/16t}
\|u\|^2_{L^2(B_{4\rho}\cap\Omega\times[0,2t])},
\end{equation*}
when $0<2t\leq T$.
The last inequality shows that \eqref{E: desigualdad de energilocal} holds
when $t$ verifies \eqref{E: condicionquecumplet} with  $N=N(n)$.
\end{proof}
\begin{lemma}\label{lemma6} There
are $\rho\in(0,1)$, $N=N(m,\varrho)\ge 1$ and  $\theta=\theta(m,\varrho)$ with $\theta\in(0,1)$, such that
\begin{equation}\label{E: vesselladesigualdad}
\|u(0)\|_{L^2(B_{\rho}\cap\Omega)}
\leq \left(Ne^{N/T}\|\tfrac{\partial u}{\partial\nu}
\|_{L^2(\triangle_8\times[0,T])}
\right)^{\theta}
\|u\|_{L^2(B_8\cap\Omega\times[0,T])}^{1-\theta},
\end{equation}
when $u$ verifies $\partial_tu+\Delta u=0$ in $B_{8}\cap\Omega\times[0,T]$, $u=0$ on $\triangle_{8}\times[0,T]$
\end{lemma}
The readers can find a similar interpolation inequality to \eqref{E: vesselladesigualdad} in \cite[Theorem 4.6]{Mdicristorondivessella} though not with optimal $T$ dependency, so that its application to observability boundary inequalities does not imply  optimal cost constants.
\begin{proof}
Let $\psi\in C_0^\infty(\mathbb{R}^n)$ and $\alpha\in C^\infty(\mathbb{R})$ be such that $\psi=1$ in $B_4$, $\psi=0$ outside $B_{5}$, $0\leq\psi\leq1$, $|\nabla \psi|+|D^2\psi|\leq N$; $
\alpha=1$ in $(-\infty,1]$, $\alpha=0$ in $[2,+\infty)$, $0\leq \alpha(t)\leq 1$ and $|\partial_t\alpha(t)|\leq 1$.
Let $\tau$ is a positive number verifying $\tfrac{4}{\tau}\leq \min\{T,1\}$. Take $h(x,t)=u(x,t)\psi(x+\rho e)\alpha(\tau t)$
in Lemma \ref{lemma2}. Here, $\rho\in(0,1)$ will be fixed later. Then,
\begin{equation*}
|\Delta h+\partial_th|\leq N\left[(1+\tau)|u|+|\nabla u|\right]\chi_{E}(x,t),
\end{equation*}
with
\begin{equation*}
E=B_5(-\rho e)\cap\Omega\times[0,\tfrac{2}{\tau}]\setminus
B_4(-\rho e)\cap\Omega\times[0,\tfrac{1}{\tau}].
\end{equation*}
Since $\frac{t}{e}\leq \sigma(t)\leq t$ in $(0,1)$, where $\sigma$ is as in Lemma \ref{lemma2}, we have
\begin{equation*}
t^{1/2}\sigma(t)^{-\tau}e^{-|x+\rho e|^2/8t}\leq
e^\tau t^{1/2-\tau}e^{-|x+\rho e|^2/8t}\\
\leq e^\tau\tau^{\tau-1/2},\ \text{when}\ (x,t)\in E.
\end{equation*}
The later inequality shows that
\begin{equation}\label{zhang3}
\begin{split}
&\|t^{1/2}\sigma(t)^{-\tau}(\Delta h+\partial_th)e^
{-|x+\rho e|^2/8t}\|_{L^2(\Omega\times(0,+\infty))}\\
&\leq N e^\tau\tau^{\tau-1/2}\left\|(1+\tau)|u|+|\nabla u|\right\|_{
L^2(B_6\cap\Omega\times[0,\frac{2}{\tau}])}.
\end{split}
\end{equation}
From Lemma \ref{lemma3} with $R=6$ and $T=\frac{2}{\tau}$, we get
\begin{equation}\label{zhang4}
\left\|(1+\tau)|u|+|\nabla u|\right\|_{L^2(B_6\cap\Omega\times[0,\frac{2}{\tau}])}
\leq N\tau \|u\|_{L^2(B_8\cap\Omega\times[0,T])}.
\end{equation}
From \eqref{zhang3} and \eqref{zhang4}, it follows that
\begin{equation}\label{zhang5}
\begin{split}
&\|t^{1/2}\sigma(t)^{-\tau}(\Delta h+\partial_th)e^
{-|x+\rho e|^2/8t}\|_{L^2(\Omega\times(0,+\infty))}\\
&\leq e^\tau\tau^{\tau+\frac{1}{2}}N
\|u\|_{L^2(B_8\cap\Omega\times[0,T])},\;\;N=N(n).
\end{split}
\end{equation}
Next, because $\partial\Omega$
is Lipschitz, there is  a positive number $\beta=\beta(m,\varrho)$ such that
$\Omega\cap B_{2\beta\rho}(-\rho e)=\emptyset$. Then
\begin{equation*}
\begin{split}
&|q+\rho e|^{1/2}\sigma(t)^{-\tau}e^
{-|q+\rho e|^2/8t}\leq \sqrt{8}\,e^\tau t^{-\tau}e^{-\frac{\beta^2\rho^2}{2t}}\\
&\leq \sqrt{8}\left(\tfrac{\beta^2\rho^2}{2}\right)^{-\tau}\tau^\tau,
\;\;\text{on}\;\;B_5(-\rho e)\cap\partial\Omega\times(0,1).
\end{split}
\end{equation*}
Thus
\begin{multline}\label{zhang6}
\| |q+\rho e|^{1/2}\sigma(t)^{-\tau}e^
{-|q+\rho e|^2/8t}\tfrac{\partial h}{\partial\nu}\|_{L^
2(\partial\Omega\times(0,+\infty))}\\
\leq \sqrt{8}\left(\tfrac{2}{\beta^2\rho^2}\right)^{\tau}\tau^\tau
\|\tfrac{\partial u}{\partial\nu}\|_{L^2(\triangle_8\times(0,T))}.
\end{multline}
From \eqref{zhang5} and \eqref{zhang6} and Lemma \ref{lemma2}, it follows that
\begin{equation}\label{zhang7}
\begin{split}
&\tau^{1/2}\|t^{-\tau}e^
{-|x+\rho e|^2/8t}u\|_{L^2(B_4(-\rho e)\cap\Omega\times[0,\frac{1}{\tau}])}\\
&\leq N\left[e^\tau\tau^{\tau+1/2}\|u\|_{L^2(B_8\cap\Omega\times
(0,T))}
+\left(\tfrac{2}{\beta^2\rho^2}\right)^{\tau}\tau^\tau
\|\tfrac{\partial u}{\partial\nu}\|_{L^2(\triangle_8\times(0,T))}\right],
\end{split}
\end{equation}
for $\tfrac{4}{\tau}\leq\min\{1,T \}$ and with $N=N(m,\varrho,n)$. Because
$$
B_{3\rho}\cap\Omega\times[\tfrac{\rho^2}{2\tau},
\tfrac{\rho^2}{\tau}]\subset B_{4\rho}(-\rho e)
\cap\Omega\times[\tfrac{\rho^2}{2\tau},
\tfrac{\rho^2}{\tau}]\subset B_{4\rho}(-\rho e)
\cap\Omega\times[0,\tfrac{1}{\tau}],
$$
we have
$$\tau^{1/2}t^{-\tau}e^
{-|x+\rho e|^2/8t}\geq \tau^{\tau+1/2}\rho^{-2\tau}e^{-4\tau},\ \text{for}\ (x,t)\in B_{4\rho}(-\rho e)
\cap\Omega\times[\tfrac{\rho^2}{2\tau},
\tfrac{\rho^2}{\tau}]. $$
Hence, the left hand side of \eqref{zhang7} verifies
\begin{equation}\label{zhang8}
\tau^{1/2}\|t^{-\tau}e^
{-|x+\rho e|^2/8t}u\|_{L^2(B_4(-\rho e)\cap\Omega\times[0,\frac{1}{\tau}])}\\
\geq \tau^{\tau+1/2}\rho^{-2\tau}e^{-4\tau}
\|u\|_{L^2(B_{3\rho}\cap\Omega\times[\frac{\rho^2}{2\tau},
\frac{\rho^2}{\tau}])}.
\end{equation}
Also, from Lemma \ref{lemma4},
\begin{equation*}\label{zhang9}
\|u\|_{L^2(B_{3\rho}\cap\Omega\times[\tfrac{\rho^2}{2\tau},
\tfrac{\rho^2}{\tau}])}
\geq \tfrac{\rho}{\sqrt{8}}\,\tau^{-1/2}\|u(0)\|_{L^2(B_\rho\cap\Omega)},
\end{equation*}
when
\begin{equation*}
\tfrac{1}{\tau}\leq
\min\left\{T,\ \tfrac{1}{N\log^{+}\left(
\frac{N\|u\|_{L^2(B_8\cap\Omega\times[0,T])}}
{\rho\|u(0)\|_{L^2(B_\rho\cap\Omega)}}
\right)}\right\}.
\end{equation*}
This, along with \eqref{zhang7} and \eqref{zhang8},
shows that
\begin{multline*}
\tau^\tau\rho^{-2\tau}e^{-4\tau}\rho\|u(0)\|
_{L^2(B_\rho\cap\Omega)}\\
\leq N\left[e^\tau\tau^{\tau+1/2}\|u\|_{L^2(B_8\cap\Omega\times[0,T])}+\left(\tfrac{2}{\beta^2\rho^2}\right)^{\tau}\tau^\tau
\|\tfrac{\partial u}{\partial\nu}\|_{L^2(\triangle_8\times(0,T))}\right],
\end{multline*}
when
\begin{equation}\label{zhang11}
\tfrac{1}{\tau}\leq\min\left\{\tfrac{1}{4},\;\;\tfrac{T}{4},\ \tfrac{1}
{N\log^{+}\left(
\frac{N\|u\|_{L^2(B_8\cap\Omega\times[0,T])}}
{\rho\|u(0)\|_{L^2(B_\rho\cap\Omega)}}
\right)}\right\}.
\end{equation}
Thus, there is $N=N(m,\varrho,n)$ such that
\begin{equation}\label{zhang10}
\rho\|u(0)\|
_{L^2(B_\rho\cap\Omega)}
\leq \left(e^7\rho^2\right)^\tau\tfrac{N}{2}
\|u\|
_{L^2(B_8\cap\Omega\times[0,T])}
+N^\tau\|\tfrac{\partial u}{\partial\nu}\|_{L^2(\triangle_8\times[0,T])},
\end{equation}
when $\tau$ satisfies (\ref{zhang11}). Fix now $\rho$ so that $e^7\rho^2=e^{-1}$
and choose
\begin{equation}\label{E: elecci—nde}
\tau=4+\tfrac{4}{T}+N\log^{+}\left(
\tfrac{N\|u\|_{L^2(B_8\cap\Omega\times[0,T])}}
{\rho\|u(0)\|_{L^2(B_\rho\cap\Omega)}}\right).
\end{equation}
Clearly, $\tau$ verifies (\ref{zhang11}). Moreover,
\begin{equation*}
\begin{split}
&\left(e^7\rho^2\right)^\tau\tfrac{N}{2}
\|u\|
_{L^2(B_8\cap\Omega\times[0,T])}
=e^{-\tau}\left(\frac{\tfrac{N}{2}
\|u\|
_{L^2(B_8\cap\Omega\times[0,T])}}{\rho\|u(0)\|_{L^2(B_\rho\cap
\Omega)}}\right)\left(\rho\|u(0)\|_{L^2(B_\rho\cap
\Omega)}\right)\\
&\leq \frac{\rho}{2}\,\|u(0)\|_{L^2(B_\rho\cap
\Omega)}.
\end{split}
\end{equation*}
This, together with  (\ref{zhang10}) shows that
\begin{equation*}
\rho\|u(0)\|
_{L^2(B_\rho\cap\Omega)}
\leq
2N^\tau\|\tfrac{\partial u}{\partial\nu}\|_{L^2(\triangle_8\times[0,T])}.
\end{equation*}
From \eqref{E: elecci—nde}, it follows that
\begin{equation*}
N^\tau=e^{\tau\log N}=e^{(4+\frac{4}{T})\log N}\left(\tfrac{N\|u\|_{L^2(B_8\cap\Omega\times[0,T])}}
{\rho\|u(0)\|_{L^2(B_\rho\cap\Omega)}}\right)^{N\log N},
\end{equation*}
and then we get that
\begin{equation*}
\begin{split}
&\left[\rho\|u(0)\|_{L^2(B_\rho\cap\Omega)}\right]^{1+N\log N}\\
&\leq e^{(4+\frac{4}{T})\log N}\left(N\|u\|_{L^2(B_8\cap\Omega\times[0,T])}\right)^{N\log N}
\|\tfrac{\partial u}{\partial\nu}\|_{L^2(\triangle_8\times[0,T])}\\
&\leq \left(Ne^{\frac{N}{T}}\|u\|_{L^2(B_8\cap\Omega\times[0,T])}\right)^{N\log N}\|\tfrac{\partial u}{\partial\nu}\|_{L^2(\triangle_8\times[0,T])}.
\end{split}
\end{equation*}
In particular,
\begin{equation*}\label{zhang13}
\begin{split}
&\|u(0)\|_{L^2(B_\rho\cap\Omega)}\\
&\leq \left(Ne^{N/T}\|\tfrac{\partial u}{\partial\nu}\|_{L^2(\triangle_8\times[0,T])}\right)^\theta
\|u\|_{L^2(B_8\cap\Omega\times[0,T])}^{1-\theta},\ \theta=\frac{1}{1+N\log N}.
\end{split}
\end{equation*}
\end{proof}

By translation and rescaling, Lemma~\ref{lemma6}
is equivalent to the following:
\begin{lemma}\label{lemma7}
Let $\Omega$ be a Lipschitz domain with constants $(m,\varrho)$, $0<R\leq 1$, $q\in\partial\Omega$, and $T>0$. Then, there are $\rho\in(0,1]$, $N=N(m,\varrho,n)\geq 1$ and $\theta=\theta(m,\varrho,n)$, with $\theta\in(0,1)$, such that
\begin{multline*}
\|u(0)\|_{L^2(B_{\rho R}(q)\cap\Omega)}\\
\leq \left(Ne^{NR^2/T}R^{\frac 12} \|\tfrac{\partial u}{\partial\nu}\|_{L^2(\triangle_{8R}(q)\times[0,T])}\right)^\theta
\,\left(R^{-1}\|u\|_{L^2(B_{8R}(q)\cap\Omega\times[0,T])}\right)^{1-\theta},
\end{multline*}
when $u$ verifies $\partial_tu+\Delta u=0$ in $B_{8R}(q)\cap\Omega\times[0,T]$ and $
u=0$ on $\triangle_{8R}(q)\times[0,T]$.
\end{lemma}
\begin{proof} [Proof of Theorem \ref{interpolation}]
It suffices to prove Theorem \ref{interpolation} when $t_2=T$, $t_1=0$ and $q=0$. From Theorem \ref{T: 2global 2-sphere 1-cylinder} or \eqref{E: dos esferas un cilindros desigualdad2}, there are $N=N(\Omega, R)$ and $0<\theta_1<1$, $\theta_1=\theta_1(\Omega, R)$ such that
\begin{equation}\label{zhang18}
\|e^{T\Delta}f\|_{L^2(\Omega)}
\leq \left(Ne^{N/T}\|e^{T\Delta}f\|_{L^2(B_{\frac{\rho R}{20}}(0',\frac{\rho R}{2}))}\right)^{\theta_1}
\|f\|_{L^2(\Omega)}^{1-\theta_1},\;\;f\in L^2(\Omega).
\end{equation}
with $\rho\in(0,1)$ as in Lemma~\ref{lemma7}. On the other hand, it follows from Lemma \ref{lemma7} that
\begin{equation*}
\begin{split}
&\|e^{T\Delta}f\|_{L^2(B_{\rho R}\cap\Omega)}\\
&\leq \left(Ne^{N/T}\|\tfrac{\partial}{\partial\nu}
\,e^{(T-t)\Delta}f\|
_{L^2(B_{8R}\cap\partial\Omega\times[0,T])}\right)^{\theta_2}\|e^{(T-t)\Delta}f\|
_{L^2(B_{8R}\cap\Omega\times[0,T])}^{1-\theta_2}.
\end{split}
\end{equation*}
In particular,
\begin{equation*}
\|e^{T\Delta}f\|_{L^2(B_{\rho R}\cap\Omega)}
\leq \left(Ne^{N/T}\|\tfrac{\partial}{\partial\nu}\,e^{t\Delta}f
\|
_{L^2(\Delta_{8R}\times[0,T])}\right)^{\theta_2}\|f\|
_{L^2(\Omega)}^{1-\theta_2}.
\end{equation*}
This, together with (\ref{zhang18}) and
$B_{\frac{\rho R}{20}}(0',\frac{\rho R}{2})\subset B_{\rho R}
\cap\Omega$, leads to the desired estimate.
\end{proof}
\begin{remark}\label{YUANYUAN1}
{\it It follows  from Theorem \ref{interpolation} that there are constants $N=N(\Omega,R,n)$ and $\theta=\theta(\Omega,n)$, $\theta\in(0,1)$, such that
\begin{equation}\label{YUANYUAN2}
\|e^{T\Delta}f\|_{L^2(\Omega)}
\leq \left(Ne^{N/[\left(\epsilon_2-\epsilon_1\right)T]}
\|\tfrac{\partial}{\partial\nu}\,e^{t\Delta}f\|
_{L^2(\Delta_{R}(q)\times[\epsilon_1T,\epsilon_2T])}\right)^{\theta}
\|f\|
_{L^2(\Omega)}^{1-\theta},
\end{equation}
when $f\in L^2(\Omega)$ and $0<\epsilon_1<\epsilon_2<1$.}
\end{remark}

Lemma \ref{L: algoque tambieninfluye} below is a rescaled and translated version of \cite[Lemma 2]{ApraizEscauriaza1}.

\begin{lemma}\label{L: algoque tambieninfluye}
Let $f$ be analytic in $[a,a+L]$ with $a\in\R$ and $L>0$, $F$ be a measurable set in $[a,a+L]$. Assume that there are positive constants $M$ and $\rho$ such that
\begin{equation}\label{E: quantitativo}
|f^{(k)}(x)|\le M k!(2\rho L)^{-k},\ \text{for}\ k\ge 0,\ x\in [a,a+L].
\end{equation}
Then, there are $N=N(\rho, |F|/L)$ and $\gamma=\gamma(\rho, |F|/L)$ with $\gamma\in (0,1)$, such that
\begin{equation*}
\|f\|_{L^\infty(a,a+L)}\le N\left(\text{\rlap |{$\int_{F}$}}\,|f|\,d\tau\right)^{\gamma}M^{1-\gamma}.
\end{equation*}
\end{lemma}

\begin{lemma}\label{L: FubiniTonelli} Let $q_0\in \partial\Omega$ and  $\mathcal J \subset \triangle_R(q_0)\times (0,T)$  be  a subset with $|\mathcal{J}|>0$.  Set
\begin{equation*}
\mathcal J_t=\{x\in \partial\Omega : (x,t)\in\mathcal J\},\ E=\{t\in (0,T): |\mathcal J_t|\ge |\mathcal J|/(2T)\},\ t\in (0,T).
\end{equation*}
Then, $\mathcal J_t\subset\triangle_R(q_0)$ is measurable for a.e. $t\in (0,T)$, $E$ is measurable in $(0,T)$, $|E|\ge |\mathcal J|/(2|\triangle_R(q_0)|)$ and
$
\chi_E(t)\chi_{\mathcal J_t}(x)\le \chi_{\mathcal J}(x,t)$ over $ \partial\Omega\times (0,T)$.
\end{lemma}

\begin{proof} From Fubini's theorem,
\begin{equation*}
|\mathcal J|=\int_0^T|\mathcal J_t|\,dt=\int_E|\mathcal J_t|\,dt+\int_{[0,T]\setminus E}|\mathcal J_t|\,dt\le |\triangle_R(x_0)||E|+|\mathcal J|/2.
\end{equation*}
\end{proof}
\begin{theorem}\label{T: casofrontera} Suppose that  $\Omega$ verifies the condition  \eqref{E: desigualdadespectralocalizada}.
Assume that  $q_0\in \partial\Omega$ and    $R\in (0,1] $ such that $\triangle_{4R}(q_0)$ is real-analytic.
Let  $\mathcal J$ be a subset in $\triangle_R(q_0)\times (0,T)$  of positive surface measure on $\partial\Omega\times (0,T)$, $E$ and $\mathcal J_t$ be the measurable sets associated to $\mathcal J$ in Lemma \ref{L: FubiniTonelli}. Then, for each $\eta\in (0,1)$,  there are $N=N(\Omega,R, |\mathcal J|/(T|\triangle_R(q_0)|),\eta)$  and $\theta=\theta(\Omega,R, |\mathcal J|/
(T|\triangle_R(q_0)|),\eta)$ with $\theta\in (0,1)$, such that the inequality
\begin{equation}\label{E: loquesaleencasofrontera}
\|e^{t_2\Delta}f\|_{L^2(\Omega)} \le \left( N e^{N/(t_2-t_1)} \int_{t_1}^{t_2}\chi_E(t) \|\tfrac{\partial}{\partial\nu}e^{t\Delta}f\|_{L^1(\mathcal J_t)}\,dt \right)^{\theta}\|e^{t_1\Delta}f\|_{L^2(\Omega)}^{1-\theta},
\end{equation}
holds, when $0\le t_1<t_2\le T$ with $t_2-t_1<1$, $|E\cap (t_1,t_2)|\ge \eta (t_2-t_1)$ and $f\in L^2(\Omega)$. Moreover,
\begin{equation}\label{E: otrajodidamierdaconlaquetira}
\begin{split}
&e^{-\frac{N+1-\theta}{t_2-t_1}}\|e^{t_2\Delta}f\|_{L^2(\Omega)}- e^{-\frac{N+1-\theta}{q\left(t_2-t_1\right)}}\|e^{t_1\Delta}f\|_{L^2(\Omega)}\\
&\le N\int_{t_1}^{t_2}\chi_E(t) \|\tfrac{\partial}{\partial\nu}\,e^{t\Delta}f\|_{L^1(\mathcal J_t)}\,dt,\;\;\mbox{when }\;\;q\ge \tfrac{N+1-\theta}{N+1}.
\end{split}
\end{equation}
\end{theorem}

\begin{proof}
Because $\triangle_{4R}(x_0)$ is real-analytic, according to Lemma \ref{L: analiticidadcalorica}, there are $N=N(\varrho,\delta)$ and $\rho=\rho(\varrho,\delta)$, $0<\rho\le1$, such that
 \begin{equation}\label{E: c13}
 |\partial_x^\alpha\partial_t^\beta e^{t\Delta}f(x)|\le\frac{N\left(t-s\right)^{-\frac n4}\, e^{8R^2/\left(t-s\right)} |\alpha|!\,\beta!}{(R\rho)^{|\alpha |}\left(\left(t-s\right)/4\right)^\beta}\,\|e^{s\Delta}f\|_{L^2(\Omega)},
\end{equation}
when $x\in \triangle_{2R}(x_0)$, $0\leq s<t$, $\alpha\in \mathbb{N}^n$ and $\beta\ge 0$.

Let $0\leq t_1<t_2\leq T$ with $t_2-t_1<1$,  $\mathcal J$ be a subset in $\triangle_R(x_0)\times (0,T)$ with positive surface measure in $\partial\Omega\times (0,T)$. Let $\mathcal J_t$ and $E$ be the measurable sets associated to $\mathcal J$ in Lemma \ref{L: FubiniTonelli}. Assume that $\eta\in (0,1)$ satisfies
\begin{equation*}
|E\cap(t_1,t_2)|\geq\eta(t_2-t_1).
\end{equation*}
Define
\begin{align*}
 &\tau=t_1+\frac{\eta}{20}(t_2-t_1),\ \tilde{t}_1=t_1+\frac{\eta}{8}(t_2-t_1),\\
&\tilde{t}_2=t_2-\frac{\eta}{8}(t_2-t_1),\ \tilde{\tau}=t_2-\frac{\eta}{20}(t_2-t_1).
\end{align*}
Then, we have
\begin{equation*}
t_1<\tau<\tilde{t}_1<\tilde{t}_2<\tilde{\tau}<t_2\ \text{and}\ |E\cap(\tilde{t}_1,\tilde{t}_2)|\geq\tfrac{3\eta}{4}(t_2-t_1).
\end{equation*}
Taking $T=t_2-t_1$ (This $T$ is different from the time $T$ in Theorem~\ref{T: casofrontera} and only used here), $\e_1=\frac\eta{20}$, $\e_2=1-\frac\eta{20}$ and  replacing $f$ by $e^{t_1\Delta}f$ in \eqref{YUANYUAN2}, we get
\begin{equation}\label{E: c14}
\|e^{t_2\Delta}f\|_{L^2(\om)}\leq Ne^{N/
\left(t_2-t_1\right)}\|\tfrac{\p}{\p\nu}
\, e^{t\Delta }f\|^\theta_{L^2(\triangle_{R}(x_0)\times(\tau,
 \;\tilde{\tau}))}\|e^{t_1\Delta}f\|
_{L^2(\Omega)}^{1-\theta},
\end{equation}
with $N=N(\Omega, R, \eta)$ and $\theta=\theta(\Omega)\in(0,1)$. From \eqref{E: c13} with $\beta=0$, $|\alpha|=1$ and $s=t_1$, there is $N=N(\Omega, R, \eta)$ such that
\begin{equation}\label{E: c15}
\|\tfrac{\p}{\p\nu}\, e^{t\Delta}f\|_{L^\infty
(\triangle_{2R}(x_0)\times(\tau,\tilde{\tau}))}
\leq Ne^{N/\left(t_2-t_1\right)}\|e^{t_1\Delta}f\|_{L^2(\om)}\ .
\end{equation}
Next, \eqref{E: c14}, the interpolation inequality
\begin{equation*}
\begin{split}
&\|\tfrac{\p}{\p\nu}\,e^{t\Delta}f\|_{L^2
(\triangle_{R}(x_0)\times(\tau,\tilde{\tau}))}\\
&\leq\|\tfrac{\p}{\p\nu}\,e^{t\Delta}f\|^{1/2}_{L^1
(\triangle_{R}(x_0)\times(\tau,\tilde{\tau}))}
\|\tfrac{\p}{\p\nu}\, e^{t\Delta}f\|^{1/2}_{L^\infty
(\triangle_{R}(x_0)\times(\tau,\tilde{\tau}))}
\end{split}
\end{equation*}
and \eqref{E: c15} yield
\begin{equation}\label{E: c16}
\|e^{t_2\Delta}f\|_{L^2(\om)}\leq Ne^{N/\left(t_2-t_1\right)}\|e^{t_1\Delta}f\|
_{L^2(\om)}^{1-\theta/2}\|\tfrac{\p}{\p\nu}\, e^{t\Delta}f\|^{\theta/2}_{L^1
(\triangle_{R}(x_0)\times(\tau,\tilde{\tau}))},
\end{equation}
with $N=N(\om, R, \eta)$ and $\theta=\theta(\Omega)\in(0,1)$. Next, setting
\begin{equation*}
v(x,t)=\tfrac{\p}{\p\nu}\,e^{t\Delta}f(x),\ \text{for}\ x\in
\triangle_{2R}(x_0),\;\;t>0,
\end{equation*}
we have
\begin{equation}\label{E: c17}
\|v\|_{L^1(\triangle_{R}(x_0)
\times(\tau,\tilde{\tau}))}\leq\left(\tilde{\tau}-\tau\right)\int_{\triangle_{R}(x_0)}
\|v(x,\cdot)\|_{L^\infty(\tau,\tilde{\tau})}\,d\s\, .
\end{equation}
Write $[\tau,\tilde{\tau}]=[a,a+L]$,  with $a=\tau$ and $L=\tilde{\tau}-\tau=(1-\frac\eta{10})(t_2-t_1)$.
Also, \eqref{E: c13} with  $|\alpha|=1$ and $s=t_1$ shows that there is $N=N(\om,R,\eta)$ such that for each fixed $x\in \triangle_{2R}(x_0)$, $t\in[\tau,\tilde{\tau}]$ and $\beta\ge 0$,
\begin{equation}\label{E: c18}
\begin{split}
|\p_t^\beta v(x,t)|&\leq
\frac{Ne^{N/\left(t-t_1\right)}\beta!}{\left(\left(t-t_1\right)/4\right)^\beta}\,\|e^{t_1\Delta}f\|_{L^2(\om)}\\
&\leq
\frac{Ne^{N/\left(t_2-t_1\right)}\beta!}{\left(\eta(t_2-t_1)/80\right)^\beta}
\|e^{t_1\Delta}f\|_{L^2(\om)}\\
&=\frac{M\beta!}{\left(2\rho L\right)^\beta}\, ,
\end{split}
\end{equation}
with
\begin{equation*}
M=Ne^{N/\left(t_2-t_1\right)}\|e^{t_1\Delta}f\|_{L^2(\om)}\ \text{and}\ \rho=\frac{\eta}{16\left(10-\eta\right)}\, .
\end{equation*}
From \eqref{E: c18}, Lemma \ref{L: algoque tambieninfluye} with $F= E\cap(\tilde{t}_1,\tilde{t}_2)$
and observing that
$$\frac{15\eta}{2\left(10-\eta\right)}\leq \frac{|F|}{L}\leq
\frac{5\left(4-\eta\right)}{2\left(10-\eta\right)},$$
we find that for each $x$ in $\triangle_{2R}(x_0)$
\begin{equation}\label{E: c19}
\|v(x,\cdot)\|_{L^{\infty}(\tau,\tilde{\tau})}
\leq \left(\text{\rlap |{$\int_{E\cap(\tilde{t}_1,\tilde{t}_2)}$}}|v(x,t)|\;dt\right)
^\gamma
\left(Ne^{N/\left(t_2-t_1\right)}\|e^{t_1\Delta}f\|_{L^2(\om)}
\right)^{1-\gamma},
\end{equation}
with $N=N(\om, R, \eta)$ and $\gamma=\gamma(\eta)$ in $(0,1)$. Thus, it follows from \eqref{E: c17} and \eqref{E: c19} that
\begin{equation}\label{E: c20}
\begin{split}
&\|v\|_{L^1(\triangle_{R}(x_0)\times
(\tau,\tilde{\tau}))}\\
&\leq\int_{\triangle_{R}(x_0)} \left(\int_{E\cap(\tilde{t}_1,\tilde{t}_2)}
|v(x,t)|\;dt\right)^\gamma d\s\left(Ne^{N/\left(t_2-t_1\right)}\,\|e^{t_1\Delta}f\|_{L^2(\om)}
\right)^{1-\gamma}\\
&\leq\left(\int_{E\cap(\tilde{t}_1,\tilde{t}_2)}
\int_{\triangle_{R}(x_0)}
|v(x,t)|\;d\s dt\right)^\gamma Ne^{N/\left(t_2-t_1\right)}\|e^{t_1\Delta}f\|^{1-\gamma}_{L^2(\om)},
\end{split}
\end{equation}
with $N$ and $\gamma$ as above. In the second inequality in \eqref{E: c20} we used H\"{o}lder's inequality.

 Next, from (\ref{E: c13}) with $\beta=0$ and $s=t_1$, we have that for $t\in (\tilde{t}_1,\tilde{t}_2)$,
 \begin{equation*}
 t-t_1\geq \tilde{t}_1-t_1=\tfrac{\eta}{8}\left(t_2-t_1\right)
 \end{equation*}
and
\begin{equation*}
\|\p_{x}^\alpha v(\cdot ,t)\|_{L^\infty(\triangle_{2R}(x_0))} \leq
\frac{Ne^{N/\left(t_2-t_1\right)} |\alpha|!}{(R\rho)^{|\alpha|}}\, \|e^{t_1\Delta}f\|_{L^2(\om)},
\end{equation*}
when $\alpha\in \mathbb{N}^n$ and with $N=N(\om,R,\eta)$.
Also, $|\mathcal{J}_t|\geq |\mathcal{J}|/\left(2T\right)$, when $t\in E$.
By the obvious generalization  of Theorem \ref{T:3} to real-analytic hypersurfaces, there are $N=N\left(\om,R, \eta, |\mathcal{J}|/\left(T|\triangle_R(x_0)|\right)\right)$ and $\vartheta=\vartheta\left(\om,|\mathcal{J}|/\left(T|\triangle_R(x_0)|\right)\right)\in(0,1)$ such that
\begin{equation}\label{E: c21}
\int_{\triangle_{R}(x_0)}|v(x,t)|\;d\s\leq
\left(\int_{\mathcal{J}_t}|v(x,t)|\;d\s\right)^\vartheta
\left(Ne^{N/\left(t_2-t_1\right)}\|e^{t_1\Delta}f\|_{L^2(\om)}
\right)^{1-\vartheta},
\end{equation}
when $t\in E\cap(\tilde{t}_1,\tilde{t}_2)$.
Both \eqref{E: c20} and (\ref{E: c21}), together with
H\"{o}lder's inequality, imply that
\begin{equation*}\label{E: c22}
\|v\|_{L^1(\triangle_{R}(x_0)\times
(\tau,\tilde{\tau}))}\leq\left(Ne^{N/\left(t_2-t_1\right)}
\int_{E\cap(\tilde{t}_1,\tilde{t}_2)}\int_{\mathcal{J}_t}
|v(x,t)|\;d\sigma dt\right)^{\vartheta\gamma}\|e^{t_1\Delta}f\|
_{L^2(\om)}^{1-\vartheta\gamma}.
\end{equation*}
This, along with \eqref{E: c16} and the definition of $v$ leads to the first estimate in this theorem.

 The  second estimate in the theorem can be proved with the method we used in the proof of the second part of Theorem \ref{T:4carcajada}.
\end{proof}

\vskip 10 pt

\begin{proof}[Proof of Theorem \ref{boundary observability}]
Let $E$ and $\mathcal J_t$ be the sets associated to $\mathcal J$ in Lemma \ref{L: FubiniTonelli} and $l$ be a density point in $E$. For $z>1$ to be fixed later, $\{l_m\}$ denotes the sequence associated to $l$ and $z$ in Lemma \ref{L: 2teoriamedida}. Because of \eqref{E: intervalosbienencajados} and from Theorem \ref{T: casofrontera} with $\eta=1/3$, $t_1=l_{m+1}$ and $t_2=l_{m}$, with $m\ge 1$,
 there are $N=N(\Omega,R, |\mathcal J|/\left(T|\triangle_R(q_0)|\right))>0$  and $\theta=\theta(\Omega,R, |\mathcal J|/\left(T|\triangle_R(q_0)|\right))$, with $\theta\in (0,1)$, such that
\begin{equation*}
\begin{split}
&e^{-\frac{N+1-\theta}{l_m-l_{m+1}}}\|e^{l_m\Delta}f\|_{L^2(\Omega)}- e^{-\frac{N+1-\theta}{q\left(l_m-l_{m+1}\right)}}\|e^{l_{m+1}\Delta}f\|_{L^2(\Omega)}\\
&\le N\int_{l_{m+1}}^{l_{m}}\chi_E(s) \|\tfrac{\partial}{\partial\nu}\,e^{s\Delta}f\|_{L^1(\mathcal J_s)}\,ds,\;\;\mbox{when}\;\;q\ge \frac{N+1-\theta}{N+1}\;\;\mbox{and}\;\;m\ge 1.
\end{split}
\end{equation*}
Let
\begin{equation*}
z=\frac{1}{2}\left(1+\frac{N+1}{N+1-\theta}\right).
\end{equation*}
Then, we can use the same arguments as those in the proof of  Theorem~\ref{T:5carcajada} to verify Theorem~\ref{boundary observability}.
\end{proof}
\begin{remark}\label{yuanyuan5}
{\it The proof of Theorem \ref{boundary observability} also implies the following observability estimate:
\begin{equation*}
\sup_{m\ge 0}\sup_{l_{m+1}\le t\le l_m}e^{-z^{m+1}A}\|e^{t\Delta}f\|_{L^2(\Omega)}\le N\int_{\mathcal J\cap (\partial\Omega\times [l,l_1])}\left|\tfrac{\partial}{\partial\nu}\,e^{t\Delta}f(x)\right|\,d\sigma dt,
\end{equation*}
for $f$ in $L^2(\Omega)$, with $A=2(N+1-\theta)^2/[\theta(l_1-l)]$ and with $z$, $N$ and $\theta$ as given along the proof of Theorem \ref{boundary observability}. Here, $l_0=T$.}
\end{remark}

\begin{remark}\label{section4remark11}
{\it $(i)$ In  Theorem \ref{boundary observability}, one may relax  more the hypothesis on $\Omega$ and to allow $\triangle_{4R}(q_0)$ to be piece piecewise analytic or simply to require that
\begin{equation*}
|\{q\in \triangle_{4R}(q_0): \triangle_{4r}(q)\ \text{is not real-analytic for some $r\le R$}\}|=0.
\end{equation*}
 $(ii)$ In particular, Theorem \ref{boundary observability} holds when $\Omega$ is a Lipschitz polyhedron in $\Rn$ and $\mathcal J$ is a measurable subset  with positive surface measure on $\partial\Omega\times (0,T)$. For if $\Omega$ is a Lipschitz polyhedron, Theorem \ref{T: unapropiedadgemotetrica} shows that  $\Omega$ verifies the condition \eqref{E: desigualdadespectralocalizada}. Also, $\mathcal J$ must have a boundary density point $(q,\tau)$, $q\in\partial\Omega$, $\tau\in (0,T)$, with $q$ in the interior of one the open flat faces of $\partial\Omega$. Thus, we can find $R>0$ such that
\begin{equation*}
\frac{|\triangle_R(q)\times(\tau-R,\tau+R)\cap\mathcal J|}{|\triangle_R(q)\times(\tau-R,\tau+R)|}\ge\frac 12\, ,
\end{equation*}
with $\triangle_{4R}(q)$ contained in a flat faces of $\partial\Omega$. Then, replace the original set $\mathcal J$ by $\triangle_R(q)\times (\tau-R,\tau+R)\cap\mathcal J\subset\triangle_R(q)\times (0,T)$, set $q_0=q$ and apply Theorem 2 as stated. $(iii)$ Theorem \ref{boundary observability} improves the work in \cite{Micu}.}
\end{remark}
\begin{remark}\label{R: 20}
{\it When $\mathcal J=\Gamma\times (0,T)$, $\Gamma\subset\triangle_{R}(q_0)$ is a measurable set, one may take $l=T/2$, $l_1=T$, $z=2$ and the constant $B$ in Theorem \ref{boundary observability} becomes
\begin{equation*}
B=A(\Omega, R, |\Gamma |/|\triangle_R(q_0)|)/T.
\end{equation*}}
\end{remark}
\begin{remark}\label{yuanyuan1}
{\it Theorem~\ref{interpolation} also implies the following: if $\Omega$ is only a bounded Lipschitz domain and verifies the condition \eqref{E: desigualdadespectralocalizada}, the heat equation is null controllable with $L^\infty(\Gamma\times (0,T))$ controls, when $\Gamma$ is an open subset of $\partial\Omega$. This seems to be a new result and we give its proof in the Appendix in section \ref{S:8}.}
\end{remark}

\section{Applications}\label{S:3}
 Throughout this section, we assume that $T>0$, $\Omega$ is a bounded Lipschitz  domain verifying  the condition  \eqref{E: desigualdadespectralocalizada} and we show several applications of Theorem \ref{T:5carcajada} and Theorem \ref{boundary observability}  to  some control problems for the heat equation.

First of all, we will show that Theorems \ref{T:5carcajada} and \ref{boundary observability} imply the null controllability with controls restricted over measurable subsets in $\Omega\times(0,T)$ and $\partial\Omega\times(0,T)$ respectively. Let $\mathcal{D}$ be a measurable subset with positive measure in $ B_{R}(x_0)\times(0,T)$ with $B_{4R}(x_0)\subset\Omega$. Let $\mathcal{J}$ be a measurable subset with positive surface measure in $ \triangle_{R}(q_0)\times(0,T)$, where $q_0\in\partial\Omega$, $R\in (0,1]$ and $\triangle_{4R}(q_0)$ is real-analytic.
Consider the following controlled heat equations:
 \begin{equation}\label{YUANYUAN3}
\begin{cases}
\partial_tu-\Delta u=\chi_{\mathcal{D}}v, &\text{in}\ \Omega\times (0,T],\\
u=0, &\text{on}\ \partial\Omega\times [0,T],\\
u(0)=u_0,\ &\text{in}\ \Omega,
\end{cases}
\end{equation}
and
 \begin{equation}\label{YUANYUAN4}
\begin{cases}
\partial_tu-\Delta u=0,\ &\text{in}\ \Omega\times (0,T],\\
u=g\,\chi_{\mathcal J},\ &\text{on}\ \partial\Omega\times [0,T],\\
u(0)=u_0,\ &\text{in}\ \Omega,
\end{cases}
\end{equation}
where $u_0\in L^2(\Omega)$,  $v\in L^\infty(\Omega\times (0,T))$ and $g\in L^\infty(\partial\Omega\times (0,T))$ are controls.
 {\it We say  that $u$ is the solution to Equation \eqref{YUANYUAN4} if
$v\equiv u-e^{t\Delta}u_0$ is the unique solution defined in  \cite[Theorem 3.2]{FabesSalsa1} (See also \cite[Theorems 8.1 and 8.3]{R1}) to
\begin{equation*}
\begin{cases}
\partial_tv-\Delta v=0,\ &\text{in}\ \Omega\times(0,T),\\
v=g\chi_{\mathcal{J}},\ &\text{on}\ \partial\Omega\times(0,T),\\
v(0)=0,\ &\text{in}\ \Omega,
\end{cases}
\end{equation*}
with $g$ in $L^p(\partial\Omega\times(0,T))$ for some $2\leq p\leq\infty$.}
 From now on, we always denote by $u(\cdot\,; u_0, v)$ and $u(\cdot\,; u_0, g)$ the solutions of Equations (\ref{YUANYUAN3}) and (\ref{YUANYUAN4}) corresponding to $v$ and $g$ respectively.

\begin{corollary}\label{cor0} For each $u_0\in L^2(\Omega)$, there are bounded control functions $v$ and $g$ with
\begin{equation*}
\|v\|_{L^\infty(\Omega\times (0,T))} \leq C_1\|u_0\|_{L^2(\Omega)},\quad  \|g\|_{L^\infty(\partial\Omega\times (0,T))}\leq C_2\|u_0\|_{L^2(\Omega)},
\end{equation*}
such that $u(T; u_0, v)=0$ and $u(T; u_0, g)=0$. Here $C_1=C(\Omega, T, R, \mathcal{D})$ and $C_2=C(\Omega, T, R, \mathcal{J})$.
\end{corollary}

\begin{proof}
 We only prove the  boundary controllability.
 Let $E$ be the measurable set associated to $\mathcal J$ in Lemma \ref{L: FubiniTonelli}. Write
\begin{equation*}
\widetilde{\mathcal{J}}=\left\{(x,t) :  (x, T-t)\in {\mathcal{J}}\right\}\;\;\mbox{and}\;\;\widetilde{E}=\left\{t :  T-t\in {E}\right\}.
\end{equation*}
  Let $l>0$ be a density point of $\widetilde{E}$ (Hence, $T-l$ is a density point of $E$).  We choose $z$, $l_1$ and the sequence $\{l_m\}$
as in the proof of Theorem \ref{boundary observability} but with $\mathcal J$ and $E$ accordingly  replaced by $\widetilde{\mathcal{J}}$ and $\widetilde{E}$. It is clear that
\begin{equation*}
0<l<\dots<l_{m+1}<l_m\dots<l_1<l_0=T,\ \lim_{m\to +\infty}l_m=l.
\end{equation*}
We set
\begin{equation*}
\mathcal M=\mathcal J\cap\left(\partial\Omega\times [T-l_1,T-l]\right)\subset\mathcal J.
\end{equation*}
It is clear that $|\mathcal{M}|>0$. The proof of Theorem~\ref{boundary observability}, the change of variables $t=T-\tau$ and Remark \ref{yuanyuan5} show that the observability inequality
\begin{equation}\label{E: observabilidadfromteraguay}
\|\varphi(0)\|_{L^2(\Omega)}\le e^B\int_{\mathcal M}|\tfrac{\partial\varphi}{\partial\nu}(p,t)|\,d\sigma dt,
\end{equation}
holds, when $\varphi$ is the unique solution in $L^\infty([0,T],L^2(\Omega))\cap L^2([0,T], H^1_0(\Omega))$ to
\begin{equation}\label{E:unadelasecuacionesmejoresuqeconozco}
\begin{cases}
\partial_t\varphi+\Delta\varphi=0,\ &\text{in}\ \Omega\times [0,T),\\
\varphi=0,\ &\text{on}\ \partial\Omega\times [0,T),\\
\varphi(T)=\varphi_T,\ &\text{in}\ \partial\Omega,
\end{cases}
\end{equation}
for some $\varphi_T$ in $L^2(\Omega)$. Set
\begin{equation*}\label{E: ladefiniciondeX22}
X=\{ \tfrac{\partial\varphi}{\partial\nu}|_{\mathcal M}: \varphi(t)=e^{(T-t)\Delta}\varphi_T,\ \text{for}\ 0\le t\le T,\ \text{for some}\ \varphi_T\in L^2(\Omega)\}.
\end{equation*}
Since $\mathcal{M}\subset\partial\Omega\times[T-l_1,T-l]$, $X$ is a subspace of $L^1(\mathcal M)$  (See \eqref{E: acojono} and \eqref{E: supernecesariamente}) and from \eqref{E: observabilidadfromteraguay}, the linear mapping $\Lambda: X\longrightarrow\R$, defined by
\begin{equation*}
\Lambda(\tfrac{\partial\varphi}{\partial\nu}|_{\mathcal M})=(u_0,\varphi(0)),
\end{equation*}
verifies
\begin{equation*}
\left|\Lambda(\tfrac{\partial\varphi}{\partial\nu}|_{\mathcal M})\right|\le e^B\|u_0\|_{L^2(\Omega)}\int_{\mathcal M}|\tfrac{\partial\varphi}{\partial\nu}(p,t)|\,d\sigma dt,\ \text{when}\  \tfrac{\partial\varphi}{\partial\nu}|_{\mathcal M}\in X.
\end{equation*}
From the Hahn-Banach theorem, there is a linear extension $T: L^1(\mathcal M)\longrightarrow\R$ of $\Lambda$, with
\begin{equation*}
\begin{split}
&T(\tfrac{\partial\varphi}{\partial\nu}|_{\mathcal M})=(u_0,\varphi(0)),\ \text{when}\ \tfrac{\partial\varphi}{\partial\nu}|_{\mathcal M}\in X,\\
&|T(f)|\le e^B\|u_0\|\|f\|_{L^1(\mathcal M)},\ \text{for all}\ f\in L^1(\mathcal M).
\end{split}
\end{equation*}
Thus, $T$ is in $L^1(\mathcal M)^{*}=L^\infty(\mathcal M)$ and there is $g$ in $L^\infty(\mathcal M)$ verifying
\begin{equation*}
T(f)=\int_{\mathcal M}fg\,d\s dt,\ \text{for all}\ f\in L^1(\mathcal M)\ \text{and}\ \|g\|_{L^\infty(\mathcal M)}\le e^B\|u_0\|.
\end{equation*}
We extend $g$ over $\partial\Omega\times (0,T)$ by setting it to be zero outside $\mathcal{M}$ and denote the extended function by $g$ again. Then it holds that $u(T;u_0, g)=0$ provided that we know that
\begin{equation}\label{E: vayasecuaciioonesmasengorosa}
\int_{\Omega}u(T;u_0, g)\varphi_T\,dx=\int_{\Omega}u_0\varphi(0)\,dx-\int_{\mathcal M}g\,\tfrac{\partial\varphi}{\partial\nu}\,d\s dt,\ \text{for all}\ \varphi_T\in L^2(\Omega).
\end{equation}

To prove (\ref{E: vayasecuaciioonesmasengorosa}), we first use the unique solvability for the problem
 \begin{equation*}
\begin{cases}
\partial_tu-\Delta u=0,\ &\text{in}\ \Omega\times (0,T],\\
u=\gamma,\ &\text{on}\ \partial\Omega\times [0,T],\\
u(0)=0\ &\text{in}\ \Omega,
\end{cases}
\end{equation*}
with lateral Dirichlet data $\gamma$ in $L^p(\partial\Omega\times (0,T))$, $2\le p\le\infty$, stablished in \cite[Theorem 3.2]{FabesSalsa1} (See also \cite[Theorems 8.1 and 8.3]{R1}). Then, because $g\chi_{\mathcal M}$ is bounded and supported in $\partial\Omega\times [T-l_1,T-l]\subset \partial\Omega\times (2\eta,T-2\eta)$ for some $\eta>0$, the calculations leading to \eqref{E: vayasecuaciioonesmasengorosa} can be justified via the regularization of $g\chi_{\mathcal M}$ and the approximation of $\Omega$ by smooth domains $\{\Omega_j;j\geq1\}$ as in \cite[Lemma 2.2]{R1}. For the sake of the completeness we provide the detailed proof of this identity in an Appendix in section \ref{S:8}.
\end{proof}

Now we apply Theorems \ref{T:5carcajada} and \ref{boundary observability} to get the bang-bang property for the minimal time control problems usually called the first type of time optimal
control problems. They are stated as follows: Let $\o$ be a measurable subset with positive measure in $B_R(x_0)$, $B_{4R}(x_0)\subset\om$.
Suppose that  $\triangle_{4R}(q_0)$ is real-analytic for some $q_0\in \partial\Omega$ and $R\in (0,1]$ and
let $\Gamma$ be a measurable subset with positive surface measure of $\triangle_R(x_0)$. For each $M>0$, define the following  control constraint set:
$$\mathcal{U}^1_M=\{v\;\;\text{measurable on}\;\; \om\times\R^+:
\;\;|v(x,t)|\leq M\;\;\text{for a.e.}\;\;(x,t)\in\om\times\R^+\}.$$
$$\mathcal{U}^2_M=\{g\;\;\text{measurable on}\;\; \partial\om\times\R^+:
\;\;|g(x,t)|\leq M\;\;\text{for a.e.}\;\;(x,t)\in\partial\om\times\R^+\}.$$
Let $u_0\in L^2(\Omega)\setminus\{0\}$.  Consider the minimal time control problems:
$$(TP)^1_M:\;\; T^1_M\equiv \min_{v\in\mathcal{U}^1_M}
\left\{t>0:\;\;e^{t\Delta}u_0+\int_0^te^{(t-s)\Delta}
({\chi_{\omega}}v)\,ds=0\right\}$$
and
$$(TP)^2_M:\;\; T^2_M\equiv \min_{g\in\mathcal{U}^2_M}
\left\{t>0:\;\;u(x,t;g)=0\;\;\mbox{for a.e.}\;\; x\in \Omega\right\},$$
where $u(\cdot, \cdot\,;g)$ is the solution to
\begin{equation}\label{ZHANG5}
\begin{cases}
\partial_t u- \Delta u=0,\ &\text{in}\ \Omega\times \R^+,\\
u=g\chi_{\Gamma},\  &\text{on}\ \partial\Omega\times\R^+,\\
u(0)=u_0,\ &\text{in}\ \Omega.
\end{cases}
\end{equation}
Any solution of $(TP)_{M}^i$, $i=1,2$, is called a minimal time control to this problem. According to Theorem~\ref{T:5carcajada} and Theorem 3.3 in  \cite{phungwangzhang}, problem $(TP)^1_M$  has solutions.  By Theorem~\ref{boundary observability}, using the same arguments as those  in the proof of Theorem 3.3 in  \cite{phungwangzhang}, we can verify that there is $g\in\mathcal{U}^2_M$  such that for some $t>0$, $u(x, t; g)=0$ for a.e.  $x\in \Omega$.

\begin{lemma}\label{ZHANG4}
Problem $(TP)^2_M $ has solutions.
\end{lemma}
\begin{proof}
Let $\{t_n\}_{n\geq 1}$, with $t_n\searrow T_M^2$, and $g_n\in \mathcal{U}^2_M$ be such that $u(x, t_n; g_n)=0$ over $\Omega$.
 Hence, on a subsequence,
\begin{equation}\label{ZHANG000}
g_n\longrightarrow g^*\;\;\mbox{weakly star in}\;\; L^\infty(\partial\Omega\times(0,t_1)).
\end{equation}
It suffices to show that
\begin{equation}\label{ZHANG00}
u_n(x,t_n)\equiv u(x,t_n; g_n)\longrightarrow u^*(x,T_M^2)\equiv u(x, T_M^2; g^*),\ \mbox{for all}\ x\in \Omega.
\end{equation}

For this purpose, let  $G(x,y,t)$ be the Green's function for $\triangle-\partial_t$ in $\Omega\times\R$ with zero lateral Dirichlet boundary condition. \cite[Theorems 1.3 and 1.4]{FabesSalsa1} and \cite[p. 643]{FabesSalsa1} show that for $g\in \mathcal{U}^2_M$ and $(x,t)\in\Omega\times (0,T)$,
\begin{equation}\label{ZHANG7}
u(x,t ;g)=e^{t\bigtriangleup}u_0- \int_{0}^{t}\int_{\partial\Omega}\tfrac{\partial G}{\partial\nu_q}(x,q,t-s)\,\chi_{\Gamma}(q, s)g(q,s)\,d\sigma_{q}ds
\end{equation}
and
\begin{equation}\label{E: controldelcuadrado}
\int_0^T\int_{\partial\Omega}|\tfrac{\partial G}{\partial\nu_q}(x,q,\tau)|^2\,d\sigma_q d\tau< +\infty,\ \text{when}\ x\in\Omega,\ T>0.
\end{equation}
 Also, by standard interior parabolic regularity there is $N=N(n, \e)$ with
\begin{equation}\label{ZHANG8}
|u(x,t;g)-u(x,s;g)|\le N|t-s|\left(\|g\|_{L^\infty(\partial\Omega\times (0,T))}+\|u_0\|_{L^2(\Omega)}\right)
\end{equation}
when $d(x,\partial\Omega)>\sqrt\e$ and $t>s\ge \e$. Now, when $x\in \Omega$ with  $d(x,\partial\Omega)>\sqrt\e$, it holds that
$$
|u_n(x,t_n)-u^*(x,T_M^2)|\leq |u_n(x,t_n)-u_n(x,T_M^2)|+ |u_n(x,T_M^2)-u^*(x,T_M^2)|.
$$
This, along with \eqref{ZHANG000}, \eqref{ZHANG7}, \eqref{E: controldelcuadrado} and \eqref{ZHANG8} indicates that \eqref{ZHANG00} holds for all $x\in\Omega$ with $d(x,\partial\Omega)>\sqrt\e$. Since $\e>0$ is arbitrary, \eqref{ZHANG00} follows at once.
\end{proof}
Now, one can use the same methods as those in \cite{gengshengwang1}, as well as in Lemma~\ref{ZHANG4},  to get the following consequences of Theorems \ref{T:5carcajada} and \ref{boundary observability} respectively:
\begin{corollary}\label{cor2}
Problem $(TP)^1_M$ has the bang-bang property: any minimal time control $v$ satisfies that $|v(x,t)|=M$ for a.e. $(x,t)\in \omega\times (0, T^1_M)$. Consequently, this problem has a
unique minimal time control.
\end{corollary}

\begin{corollary}\label{cor4}
The problem $(TP)^2_M$ has the bang-bang property: any minimal time boundary control $g$ satisfies that $|g(x,t)|=M$ for a.e. $(x,t)\in \Gamma\times (0, T^2_M)$. Consequently, this problem has a
unique minimal time control.
\end{corollary}

Next, we make use of  Theorems~\ref{T:5carcajada} and \ref{boundary observability} to study the bang-bang property for the time optimal control problems where the interest is on retarding the initial time of the action of a control with bounded $L^\infty$-norm. These problems are usually called the second type of time optimal control problems and are
stated as follows: Let $T>0$ and $M>0$.
 Write  $\omega$ and $\Gamma$ for the sets given in
 Problems $(TP)_M^1$ and $(TP)_M^2$ respectively. Consider the controlled heat equations:
\begin{equation}\label{ZHANG1}
\begin{cases}
\partial_tu- \Delta u=\chi_{\omega}\chi_{(\tau,T)}v, &\text{in}\ \Omega\times (0,T],\\
u=0, &\text{on}\ \partial\Omega\times [0,T],\\
u(0)=u_0,\ &\text{in}\ \Omega
\end{cases}
\end{equation}
and
\begin{equation}\label{ZHANG2}
\begin{cases}
\partial_tu-\Delta u=0, &\text{in}\ \Omega\times (0,T],\\
u=\chi_{\Gamma}\chi_{(\tau,T)}g, &\text{on}\ \partial\Omega\times [0,T],\\
u(0)=u_0,\ &\text{in}\ \Omega,
\end{cases}
\end{equation}
where $u_0\in L^2(\Omega)$. Write accordingly  $u(\cdot \,; \chi_{(\tau,T)}v)$ and $u(\cdot\, ; \chi_{(\tau,T)}g)$ for  the solutions  to equation (\ref{ZHANG1})
corresponding to $\chi_{(\tau,T)}v$, and to equation (\ref{ZHANG2}) corresponding to $\chi_{(\tau,T)}g$.
Define the following control constraint sets:
$$\mathcal{U}^1_{M,T}=\{v\ \text{measurable on}\ \om\times(0,T): |v(x,t)|\leq M\ \text{for a.e.}\ (x,t)\in\om\times(0,T)\}.$$
$$\mathcal{U}^2_{M,T}=\{g\ \text{measurable on}\ \partial\om\times(0,T): |g(x,t)|\leq M\ \text{for a.e.}\ (x,t)\in\partial\om\times(0,T)\}.$$
Consider the time optimal control problems:
$$(TP)^1_{M,T}:\ \tau^1_{M,T}\equiv \sup_{v\in\mathcal{U}^1_{M,T}}
\left\{\tau\in[0,T):\;\;u(T; \chi_{(\tau,T)}v)=0\right\}$$
and
$$(TP)^2_{M,T}:\;\; \tau^2_{M,T}\equiv \sup_{g\in\mathcal{U}^2_{M,T}}
\left\{\tau\in[0,T):\;\;u(T; \chi_{(\tau,T)}g)=0\right\}.$$
Any solution of $(TP)_{T,M}^i$, $i=1,2$, is called an optimal control to the corresponding problem.

Now, we can use  the same arguments as those in the proof of Theorem 3.4 in \cite{PhungWang1} to get the following consequences
of Theorem~\ref{T:5carcajada} and Theorem~\ref{boundary observability} respectively:

\begin{corollary}\label{cor5} Any optimal control $v^*$ to Problem $(TP)^1_{M,T}$, if it exists, satisfies the bang-bang property: $|v^*(x,t)|=M$ for a.e. $(x,t)\in \omega\times(\tau^1_{M,T}, T)$.
\end{corollary}
\begin{corollary}\label{cor6} Any optimal control $g^*$ to Problem  $(TP)^2_{M,T}$, if it exists, satisfies the bang-bang property: $|g^*(x,t)|=M$ for a.e. $(x,t)\in \Gamma\times(\tau^2_{M,T}, T)$.
\end{corollary}

\begin{remark}\label{ZHANG3} {\it By Theorem~\ref{T:5carcajada} (See also Remark \ref{R: 3}) and the energy decay property for the heat equation, one can easily prove the following:  for a fixed $M>0$, there is $v\in \mathcal{U}^1_{M,T}$
 such that $u(T;\chi_{(0, T)}v)=0$,  when $T$ is large enough (such a control $v$ is called an admissible control); while for a fixed $T>0$, the same holds when $M$ is large enough.
 The same can be said about  Problem $(TP)^2_{M,T}$ because of Theorem~\ref{boundary observability} (See also Remark~\ref{R: 20}).
In the case where Problem $(TP)^1_{M,T}$ has admissible controls, one can easily prove the existence of time optimal controls to this problem. In the case when Problem $(TP)^2_{M,T}$ has admissible controls, one can make use of the similar method in the proof of Lemma~\ref{ZHANG4} to verify the existence of time optimal controls for this problem.}
\end{remark}

Finally, we utilize  Theorem~\ref{T:5carcajada} and   Theorem~\ref{boundary observability} to study the bang-bang property for the minimal norm control problems, which are stated as follows: Let $\mathcal{D}$ and  $\mathcal{J}$ be the subsets given at the beginning of this section.
Let  $u_0\in L^2(\Omega)$. Define two control constraint sets as follows:
$$\mathcal{V}_{\mathcal{D}}=\left\{v\in L^\infty (\Omega\times(0,T)): u(T; u_0, v)=0 \right\}$$
and
$$\mathcal{V}_{\mathcal{J}}=\left\{g\in L^\infty (\partial\Omega\times(0,T)):\;u(T;u_0,g)=0 \right\}.$$
Consider the minimal norm control problems:
$$(NP)_{\mathcal{D}}:\;\;\;\;M_{\mathcal{D}}\equiv \min\left\{\|v\|_{L^\infty(\Omega\times(0,T))}:\;\;
v\in\mathcal{V}_{\mathcal{D}}\right\}$$
and
$$(NP)_{\mathcal{J}}:\;\;\;\;M_{\mathcal{J}}\equiv \min\left\{\|g\|_{L^\infty(\partial\Omega\times(0,T))}:\;\;
g\in\mathcal{V}_{\mathcal{J}}\right\}.$$
Any solution of $(NP)_{\mathcal{D}}$ (or $(NP)_{\mathcal{J}}$) is called a minimal norm control to this problem.  According to Corollary~\ref{cor0}, the sets  $\mathcal{V}_{\mathcal{D}}$ and  $\mathcal{V}_{\mathcal{J}}$ are not empty. Since $\mathcal{V}_{\mathcal{D}}$ is not empty, it follows from  the standard arguments that Problem $(NP)_{\mathcal{D}}$ has solutions. Because $\mathcal{V}_{\mathcal{J}}$ is not empty,
by using the similar arguments as those in the proof of
Lemma~\ref{ZHANG4}, we can justify that Problem $(NP)_{\mathcal{J}}$
has solutions. Now, one can use the same methods as those in  \cite{PhungWang1}  to get the following consequences of Theorem \ref{T:5carcajada} and Theorem~\ref{boundary observability} respectively:
\begin{corollary}\label{cor1} Problem $(NP)_{\mathcal{D}}$ has the bang-bang property: any minimal norm control $v$ satisfies that $|v(x,t)|=M_{\mathcal{D}}$ for a.e. $(x,t)\in \mathcal{D}$. Consequently, this problem has a
unique minimal norm control.
\end{corollary}
\begin{corollary}\label{cor3} The problem $(NP)_{\mathcal{J}}$ has the bang-bang property: any minimal norm boundary-control $g$ satisfies that $|g(x,t)|=M_{\mathcal{J}}$ for a.e. $(x,t)\in \mathcal{J}$. Consequently, this problem has a
unique minimal norm control.
\end{corollary}
\noindent {\it Acknowledgement}: The authors want to thank G. Verchota for showing them that not all the polygons or polyhedrons are Lipschitz domains.
\vspace{1. cm}
\section{Appendix}\label{S:8}
\begin{proof}[Proof of (\ref{E: vayasecuaciioonesmasengorosa})] For each $(p,\tau)\in\partial\Omega\times \R$ and fixed $\xi>0$,  we define
\begin{equation*}
\Gamma(p)=\{x\in\Omega: |x-p|\le\left(1+\xi\right)d(x,\partial\Omega)\},
\end{equation*}
\begin{equation*}
\Gamma(p,\tau)=\{(x,t)\in\Omega\times (0,T): |x-p|+\sqrt{|t-\tau|}\le\left(1+\xi\right)d(x,\partial\Omega)\}.
\end{equation*}
The later  are called respectively elliptic and parabolic non-tangential approach regions from the interior of $\Omega\times (0,T)$ to $(p,\tau)$. In particular,
\begin{equation*}
\Gamma(p)\times\{\tau\}\subset\Gamma(p,\tau),\ \text{for all}\  (p,\tau)\in\partial\Omega\times (0,T).
\end{equation*}
When $u:\Omega\longrightarrow\R$ or $u:\Omega\times (0,T)\longrightarrow\mathbb{R}$ (or $\mathbb{R}^n$), define the elliptic and parabolic non-tangential maximal function of $u$ in $\partial\Omega\times (0,T)$ as
\begin{equation*}
u^\ast(p)=\sup_{x\in\Gamma(p)}|u(x)|,\quad u^\sharp(p,\tau)=\sup_{(x,t)\in\Gamma(p,\tau)}|u(x,t)|,\ \text{when}\ p\in\partial\Omega\ \text{and}\ \tau\in (0,T).
\end{equation*}

Let $\eta>0$ be fixed such that $[T-l_1,T-l]\subset[2\eta,T-2\eta]$, with $l$ and $l_1$ as defined in Corollary \ref{cor0}. Denote by $u$ the solution to
\begin{equation*}
\begin{cases}
\partial_tu-\Delta u=0,\ &\text{in}\ \Omega\times(0,T),\\
u=g\chi_{\mathcal{M}}\equiv\gamma,\ &\text{on}\ \partial\Omega\times(0,T),\\
u(0)=u_0,\ &\text{in}\ \Omega.
\end{cases}
\end{equation*}
(See the beginning of Section 5 for  the definition of the solution to this equation.)

Let $\gamma^\varepsilon$ in $C^1_0(\partial\Omega\times(0,T))$
 be a regularization of $\gamma$ in $\partial\Omega\times[0,T]$ such that
\begin{equation*}
\|\gamma^\varepsilon\|_{L^\infty(\partial\Omega\times[0,T])}+\e\,\|\gamma^\varepsilon\|_{C^1(\partial\Omega\times[0,T])}
\leq \|\gamma\|_{L^\infty(\partial\Omega\times[0,T])},
\end{equation*}
\begin{equation*}
\text{supp}(\gamma^\varepsilon)\subset\partial\Omega\times[\eta,T-\eta]
\end{equation*}
and let $v^{\varepsilon}$ be the solution to
\begin{equation*}
\begin{cases}
\partial_tv^{\varepsilon}-\Delta v^{\varepsilon}=0,\ &\text{in}\ \Omega\times(0,T),\\
v^{\varepsilon}=\gamma^{\varepsilon},\ &\text{on}\ \partial\Omega\times(0,T),\\
v^{\varepsilon}(0)=0,\ &\text{in}\ \Omega.
\end{cases}
\end{equation*}
From \cite[Theorem 3.2]{FabesSalsa1} and either \cite[Theorem 6.1]{R1} or \cite[Theorem 2.9]{R2}
\begin{equation}\label{green0}
\|v^\varepsilon\|_{L^\infty(\partial\Omega\times[0,T])}+\e\,\|\left(\nabla v^\e\right)^\sharp\|_{L^2(\partial\Omega\times [0,T])}\le\|\gamma\|_{L^\infty(\partial\Omega\times[0,T])},
\end{equation}
and the limits
\begin{equation*}
 \lim_{\underset{(x,t)\in\Gamma(p,\tau)}{(x,t)\rightarrow (p,\tau)}}\nabla v^\e(x,t)=
\nabla v^\e(p,\tau)
\end{equation*}
exist and are finite for a.e. $(p,\tau)$ in $\partial\Omega\times (0,T)$.
Also, $v_\e\in C(\overline\Omega\times [0,T])\cap C^\infty(\Omega\times [0,T])$, $v^\varepsilon=0$ for $t\leq \eta$,
and $v^\varepsilon=0$ on $\partial\Omega\times(T-\eta,T]$. Moreover, the H\"older regularity up to the boundary for bounded solutions to parabolic equations with zero local  lateral Dirichlet data, shows that there are positive constants $N=N(m,\varrho,\eta)$
and $\alpha=\alpha(m,\varrho)$, with $\alpha\in(0,1)$,  such that
\begin{equation}\label{green2}
\begin{split}
&|v^\varepsilon(x_1,t_1)-v^\varepsilon(x_2,t_2)|
\leq N\left[|x_1-x_2|^2+|t_1-t_2|\right]^{\alpha/2}
\|\gamma\|_{L^\infty(\partial\Omega\times[0,T])},
\end{split}
\end{equation}
when $x_1,x_2\in\overline{\Omega},\;T-\frac\eta 2\leq t_1,t_2\leq T$ \cite[Theorems 6.28 and 6.32]{Lieberman1}.

Let $\varphi(t)=e^{(T-t)\Delta}\varphi_T$, $t\in(0,T)$, where $\varphi_T$ is in $L^2(\Omega)$. From the regularity of caloric functions \cite[Theorem 1.7]{F1}
\begin{equation}\label{green5}
\varphi\in C([0,T];L^2(\Omega))\cap
C^\infty(\Omega\times[0,T))\cap C(\overline{\Omega}\times[0,T))
\end{equation}
and from \cite[Theorems 1.3 and 1.4]{FabesSalsa1} or the proof of (\ref{E: acojono}) and (\ref{E: supernecesariamente}) in this appendix,
 there are $N=N(m,\varrho)$ and  $\e=\e(m,\varrho,n)>0$ such that
\begin{equation}\label{E: acojono}
\|(\nabla \varphi)^\ast\|
_{L^\infty(0,T-\delta;\, L^{2+\e}(\partial\Omega))}
\leq Ne^{1/\delta}\,\|\varphi_T\|_{L^2(\Omega)},
\end{equation}
when $0<\delta<T$ and the limit
\begin{equation}\label{E: supernecesariamente}
\lim_{\underset{x\in\Gamma(p)}{x\to p}}\nabla\varphi(x,\tau)=
\nabla\varphi(p,\tau),
\end{equation}
exists and is finite  for a.e. $p\in\partial\Omega$ and for all $\tau\in (0,T)$.
Now, let $\Omega_j\subset\overline\Omega_{j+1}\subset\Omega$, $j\geq 1$, be a sequence of $C^\infty$-domains approximating $\Omega$ as in \cite[Lemma 2.2]{R1}. Set, $u^\varepsilon=v^\varepsilon+
e^{t\Delta}u_0$. By Green's formula,
\begin{equation*}
\tfrac{d}{dt}\int_{\Omega_j}u^\varepsilon(t)\varphi(t)\,dx
=
\int_{\partial\Omega_j}\tfrac{\partial u^\varepsilon}{\partial\nu_j}\,\varphi-\tfrac{\partial \varphi}{\partial\nu_j}\, u^\varepsilon\,d\sigma_j.
\end{equation*}
Integrating the above identity over $[\delta,T-\delta]$ for a fixed $\delta\in(0,\frac{\eta}{2})$, we get
\begin{equation}\label{equality1}
\begin{split}
&\int_{\Omega_j}u^\varepsilon(T-\delta)\varphi(T-\delta)\,dx
-\int_{\Omega_j}u^\varepsilon(\delta)\varphi(\delta)\,dx\\
&=
\int_{\partial\Omega_j\times(\delta,T-\delta)}\tfrac{\partial u^\varepsilon}{\partial\nu_j}\,\varphi-\tfrac{\partial \varphi}{\partial\nu_j}\, u^\varepsilon\,d\sigma_jdt.
\end{split}
\end{equation}
Recall that $u^\varepsilon(\delta)=e^{\delta\Delta}u_0$ and let $j\rightarrow+\infty$ in \eqref{equality1} with $\epsilon$ and $\delta$ being fixed. Then, \eqref{green0}, \eqref{green5}, \eqref{E: supernecesariamente} and the dominated convergence theorem show that\begin{equation*}
\int_{\Omega}u^\varepsilon(T-\delta)\varphi(T-\delta)\,dx
=\int_{\Omega}(e^{\delta\Delta}u_0)\varphi(\delta)\,dx
-\int_{\partial\Omega\times(\delta,T-\delta)}\gamma^\varepsilon\tfrac{\partial \varphi}{\partial\nu}\,d\sigma dt.
\end{equation*}
Because $\gamma^\varepsilon$ is supported in $[\eta,T-\eta]$,
the later is the same as
\begin{equation}\label{equality2}
\int_{\Omega}u^\varepsilon(T-\delta)\varphi(T-\delta)\,dx
=\int_{\Omega}(e^{\delta\Delta}u_0)\varphi(\delta)\,dx
-\int_{\partial\Omega\times(\eta,T-\eta)}\gamma^\varepsilon\tfrac{\partial \varphi}{\partial\nu}\,d\sigma dt,
\end{equation}
when $0<\delta<\eta/8$. Next, from \eqref{green2},
\begin{equation*}
u^\varepsilon(T-\delta)=v^\varepsilon(T-\delta)+e^{(T-\delta)
\Delta}u_0=v^\varepsilon(T)+e^{T\Delta}u_0+O(\delta^{\alpha/2}),
\end{equation*}
uniformly for $x\in\overline{\Omega}$, when $0<\delta<\eta/8$. Hence, after letting $\delta\rightarrow 0$ in \eqref{equality2}, we get
\begin{equation*}
\int_{\Omega}(v^\varepsilon(T)+e^{T\Delta}u_0)\varphi(T)\,dx
=\int_{\Omega}u_0\varphi(0)\,dx-\int_{\partial\Omega\times(\eta
,T-\eta)}\gamma^\varepsilon\tfrac{\partial\varphi}{\partial\nu}\,
d\sigma dt.
\end{equation*}

Also, from \eqref{green0} and \eqref{green2},
$v^\varepsilon$ converges uniformly over $\overline{\Omega}
\times[T-\eta/2,T]$ to some continuous function $\widetilde{v}$ as $\varepsilon\rightarrow 0$. We  claim that
$\widetilde{v}=v$. If it is the case, we get
after letting $\varepsilon\rightarrow 0$ in the last equality, that
\begin{equation*}
\int_{\Omega}u(T)\varphi(T)\,dx
=\int_{\Omega}u_0\varphi(0)\,dx-\int_{\partial\Omega\times(\eta
,T-\eta)}\gamma\tfrac{\partial\varphi}{\partial\nu}\,
d\sigma dt,
\end{equation*}
because $\gamma^\varepsilon(p,\tau)\rightarrow\gamma(p,\tau)$
for a.e. $(p,\tau)\in\partial\Omega\times(0,T)$, \eqref{green5} and
 \begin{equation*}
 \text{supp}(\gamma^\varepsilon)\cup \text{supp}(\gamma)\subset
\partial\Omega\times[\eta,T-\eta].
\end{equation*}
Recalling that $\gamma=g\chi_{\mathcal{M}}$, we get
\begin{equation*}
\int_{\Omega}u(T)\varphi(T)\,dx
=\int_{\Omega}u_0\varphi(0)\,dx-\int_{\partial\Omega\times(0
,T)}g\chi_{\mathcal{M}}\,\tfrac{\partial\varphi}{\partial\nu}\,
d\sigma dt.
\end{equation*}
Hence, \eqref{E: vayasecuaciioonesmasengorosa} is proved.

To verify that $\widetilde{v}=v$ over $\overline{\Omega}\times[0,T]$, observe that because $v^\e-v$ is the unique solution to
\begin{equation*}
\begin{cases}
\partial_tu-\Delta u =0,\ &\text{in}\ \Omega\times(0,T),\\
u=\gamma^{\e}-\gamma,\ &\text{on}\ \partial\Omega\times(0,T),\\
u(0)=0,\ &\text{in}\ \Omega,
\end{cases}
\end{equation*}
whose parabolic non-tangential maximal function is in $L^2(\partial\Omega\times(0,T))$ (See \cite[Theorem 3.2]{FabesSalsa1}), it holds that
\begin{equation}\label{E: acotaciontrivial}
\|(v^\varepsilon-v)^\sharp\|_{L^2(\partial\Omega\times(0,T))}
\leq N\|\gamma^\varepsilon-\gamma\|_{L^2(\partial\Omega\times(0,T))}.
\end{equation}
For fixed $p$ in $\partial\Omega$, we may assume that  $p=(0',0)$
and that near $p$,
\begin{equation*}
\Omega\cap Z_{m,\varrho}=\{(x',x_n):\phi(x')< x_n< 2m\varrho,\; |x'|\leq \varrho\},
\end{equation*}
with $\phi$ as in \eqref{E: condicionLipschitz} and \eqref{E: segunacondicionlipschitz}. Then,
\begin{equation*}
\begin{split}
&\int_0^T\int_{B'_{\varrho}}\int_{\phi(y')}^{\phi(y')+m\varrho}
|F(y',y_n,t)|^2\,dy'dy_ndt\\
&\leq m\varrho\int_0^T\int_{B'_\varrho}F^\sharp(y',y_n,t)^2\,dy'dt
\leq m\varrho\int_{\partial\Omega\times(0,T)}F^\sharp(p,t)^2\,d\sigma
dt,
\end{split}
\end{equation*}
for all functions $F$. The above estimate, a covering argument and \eqref{E: acotaciontrivial} show that
\begin{equation}\label{green6}
\|v^\e-v\|_{L^2(\Omega_{m\varrho}\times (0,T))}\le N\|\gamma^\e-\gamma\|_{L^2(\partial\Omega\times(0,T))},
\end{equation}
with $\Omega_\eta=\{x\in\Omega: d(x,\partial\Omega)\le\eta\}$. Recalling that $v^\e=v=0$ for $t\le\eta$, the local boundedness properties of solutions to parabolic equations \cite[Theorem 6.17]{Lieberman1}  show that,
\begin{equation*}
|(v^\varepsilon-v)(x,\tau)|\leq \left(-\!\!\!\!\!\!
\int_{B_{\frac{R}{20}}(x)
\times[\tau-\frac{R^2}{20^2},\tau]}|v^\varepsilon-v|^2\,dyds
\right)^{1/2},
\end{equation*}
when $x\in\partial\Omega_{R}$, $0\le\tau\leq T$, and taking $R<\frac{m\varrho}{20}$ above, we find from \eqref{green6} that
\begin{equation*}
\|v^\varepsilon-v\|_{L^\infty(\Omega^R\times\{0\}\cup\partial\Omega
^R\times[0,T])}
\leq N_R\left\|\gamma^\varepsilon-\gamma\right\|_{L^2(\partial
\Omega\times(0,T))}.
\end{equation*}
By the maximum principle and the above estimate
\begin{equation*}
\|v^\varepsilon-v\|_{L^\infty(\Omega^R\times[0,T])}
\leq N_R\left\|\gamma^\varepsilon-\gamma\right\|_{L^2(\partial
\Omega\times(0,T))}\longrightarrow 0,\;\;\text{as}\;\;\varepsilon\rightarrow 0,
\end{equation*}
which shows that $\widetilde{v}=v$ in $\overline{\Omega}\times[0,T]$.
\end{proof}
\vskip 10 pt

Next we give the proof of  (\ref{E: acojono}) and (\ref{E: supernecesariamente}). For this purpose, we first need to recall the following known result which follows from \cite[Theorem 3, Lemmas 1 and 8]{Dahlberg1} (See also \cite[Theorem 5.19]{KenigPipher}):
\begin{lemma}\label{L: integrabiidad} Let $\Omega$ be a Lipschitz domain in $\Rn$, $p\in\partial\Omega$ and $u$ in $C(\overline{\Omega}\cap B_{2R}(p))$ verify $\Delta u=0$ in $\Omega\cap B_{2R}(p)$ and $u=0$ on $\partial\Omega\cap B_{2R}(p)$. Then, there are  $N=N(m,\varrho)$ and $\e=\e(m,\varrho,n)$, $0<\e<1$, such that
\begin{equation*}
 \|(\nabla u)^*\|_{L^{2+\e}(\triangle_{R}(p))}
 \leq N R^{-1-\frac n2+\frac{n-1}{2+\e}}\|u\|_{L^2(B_{2R}(p)\cap\Omega)}.
 \end{equation*}
 Moreover, the limit
 \begin{equation*}
\nabla u(q)=\lim_{\underset{x\in\Gamma(q)}{x\rightarrow q}}\nabla u(x),\ \text{exists and is finite for a.e.}\  q\in \triangle_R(p).
\end{equation*}
\end{lemma}

\begin{proof}[ Proof of (\ref{E: acojono}) and (\ref{E: supernecesariamente})]

A covering of the lateral boundary of $\Omega\times (-\tfrac 12,\tfrac 12)$ and the application of Lemma \ref{L: integrabiidad} to the harmonic functions $u_j(x,y)=e^{\sqrt{\lambda_j}\,y}e_j(x)$ with $R=\tfrac 12$, $j\ge 1$, show that
\begin{equation}\label{E: loqueloescocurrealasautofuncioens}
\|\left(\nabla e_j\right)^\ast\|_{L^{2+\e}(\partial\Omega)}\le Ne^{\sqrt{\lambda_j}}\quad \text{and}\quad \lim_{\underset{x\in\Gamma(p)}{x\to p}}\nabla e_j(x)=\nabla e_j(p),
\end{equation}
exists and is finite for a.e. $p\in\partial\Omega$ and  for all $j\ge 1$. Recall that
\begin{equation*}
\varphi(t)=e^{\left(T-t\right)\triangle}\varphi_T=\sum_{j\geq 1}e^{-\lambda_j(T-t)}\left(\varphi_T,e_j\right) e_j,\;\;t\in[0,T],
\end{equation*}
when $\varphi_T$ is in $L^2(\Omega)$. Then,  for $(p,\tau)\in\partial\Omega\times [0,T-\delta]$, $\delta>0$ and $x\in \Gamma(p)$ with $t\le T-\delta$, we have
\begin{equation*}
|\nabla \varphi(x,t)|\le\sum_{j\geq 1}e^{-\lambda_j\delta}|(\varphi_T,e_j)|\left(\nabla e_j\right)^\ast(p).
\end{equation*}
Thus,
\begin{equation*}
\left(\nabla\varphi\right)^\ast(p,\tau)\le \sum_{j\geq 1}e^{-\lambda_j\delta}|(\varphi_T,e_j)|\left(\nabla e_j\right)^\ast(p),
\end{equation*}
and from \eqref{E: loqueloescocurrealasautofuncioens}
\begin{equation*}
\begin{split}
&\|(\nabla\varphi)^\ast(\cdot,\tau)\|_{L^{2+\e}(\partial\Omega)}\leq \sum_{j\geq 1}e^{-\lambda_j\delta}|\left(\varphi_T,e_j\right)| \|(\nabla e_j)^*
\|_{L^{2+\e}(\partial\Omega)}\\
&\leq N\sum_{j\geq 1}e^{-\lambda_j\delta+\sqrt{\lambda_j}}|\left(\varphi_T,e_j\right)|\leq Ne^{1/\delta}\sum_{j\geq 1}e^{-\lambda_j\delta/2}|\left(\varphi_T,e_j\right)|\\
&\leq Ne^{1/\delta}\Big(\sum_{j\geq 1}|\left(\varphi_T,e_j\right)|^2\Big)^{1/2}
\Big(\sum_{j\geq 1}e^{-\lambda_j\delta}\Big)^{1/2}\\
&=Ne^{1/\delta}\|\varphi_T\|_{L^2(\Omega)}
\Big(\sum_{j\geq 1}e^{-\lambda_j\delta}\Big)^{1/2}
\end{split}
\end{equation*}
Now, integrate \eqref{E: algo que ayuda} over $\Omega$ to find that
\begin{equation*}\label{E: otradesigualdad}
\sum_{j\geq 1}e^{-\lambda_j\delta}\le \left(4\pi\delta\right)^{-\tfrac n2}|\Omega|
\end{equation*}
and get \eqref{E: acojono}. Next, for $i=1,\dots,n$, $\varphi_k=\sum_{j\le k}e^{-\left(T-t\right)\lambda_j}\left(\varphi_T,e_j\right)e_j$, we have
\begin{equation*}
\begin{split}
&|\{p\in\partial\Omega: \limsup_{\underset{x\in\Gamma(p)}{x\to p}}\partial_i\varphi(x,t)-\liminf_{\underset{x\in\Gamma(p)}{x\to p}}\partial_i\varphi(x,t)>\lambda\}|\\
&=|\{p\in\partial\Omega: \limsup_{\underset{x\in\Gamma(p)}{x\to p}}\left(\partial_i\varphi-\partial_i\varphi_k\right)(x,t)-\liminf_{\underset{x\in\Gamma(p)}{x\to p}}\left(\partial_i\varphi-\partial_i\varphi_k\right)(x,t)>\lambda\}|\\
&\le |\{p\in\partial\Omega: \left(\nabla\varphi-\nabla\varphi_k\right)^\ast(p,t)>\lambda/2\}|\\
&\le \tfrac {4N^2 e^{2/\delta}}{\lambda^2} \sum_{j> k}\left(\varphi_T,e_j\right)^2,\ \text{when}\ t\le T-\delta,
\end{split}
\end{equation*}
which shows that \eqref{E: supernecesariamente} holds letting $k$ tend to infinity.
\end{proof}

\begin{proof}[Proof of Remark \ref{yuanyuan1}]
From the estimate in Theorem \ref{interpolation} with $\triangle_R(q)\subset\Gamma$,
\begin{equation*}
\|e^{t_2\Delta}f\|_{L^2(\Omega)}\leq \left(Ne^{\frac{N}{t_2-t_1}}
\|\tfrac{\partial}{\partial\nu}\,e^{t\Delta}f\|_{L^2(\Gamma\times
(t_1,t_2))}\right)^\theta\|e^{t_1\Delta}f\|^{1-\theta}_{L^2(\Omega)},
\;f\in L^2(\Omega)
\end{equation*}
 and the telescoping series method, we can get the following
$L^2$-observability inequality:
\begin{equation}\label{i3}
\|e^{L\Delta}f\|_{L^2(\Omega)}\leq Ne^{N/L}
\|\tfrac{\partial}{\partial\nu}\,
e^{t\Delta}f\|_{L^2(\Gamma\times
(L/2,L))},\;\;L\in(0,T).
\end{equation}
Next, recall the $L^p$-interpolation inequality,
\begin{equation}\label{i2}
\|\tfrac{\partial}{\partial\nu}\,e^{t\Delta}f
\|_{L^2(\Gamma\times(L/2,L))}
\leq \|\tfrac{\partial}{\partial\nu}\, e^{t\Delta}f
\|_{L^1(\Gamma\times(L/2,L))}
^{\frac{\e}{2\left(1+\e\right)}}
\|\tfrac{\partial}{\partial\nu}\, e^{t\Delta}f\|_{L^
{2+\varepsilon}
(\Gamma\times(L/2,L))}^{\frac{2+\e}{2\left(1+\e\right)}},
\end{equation}
and the bound
\begin{equation}\label{i1}
\|\tfrac{\partial}{\partial\nu}\,e^{t\Delta}f\|_{L^{2+\epsilon}
(\Gamma\times(L/2,L))}
\leq Ne^{N/L}\|f\|_{L^2(\Omega)},
\end{equation}
which follows from  \eqref{E: acojono} with $T=L$ and $\delta=L/2$. Then, from \eqref{i1}, \eqref{i2} and \eqref{i3},
\begin{equation*}
\begin{split}
\|e^{L\Delta}f\|_{L^2(\Omega)}
&\leq \left(Ne^{N/L}\|\tfrac{\partial}
{\partial\nu}\,e^{t\Delta}f\|_{L^{1}
(\Gamma\times(L/2,L))}\right)^\rho
\|f\|_{L^2(\Omega)}^{1-\rho}\\
&\leq\left(Ne^{N/L}\|\tfrac{\partial}
{\partial\nu}\,e^{t\Delta}f\|_{L^{1}
(\Gamma\times(0,L))}\right)^\rho
\|f\|_{L^2(\Omega)}^{1-\rho}\ ,
\end{split}
\end{equation*}
with $\rho=\frac{\e}{2\left(1+\e\right)}$\,.
In particular,
\begin{equation*}
\begin{split}
&\|e^{t_2\Delta}f\|_{L^2(\Omega)}
\leq\left(Ne^{N/\left(t_2-t_1\right)}\|\tfrac{\partial}
{\partial\nu}\,e^{t\Delta}f\|_{L^{1}
(\Gamma\times(t_1,t_2))}\right)^\rho
\|e^{t_1\Delta}f\|_{L^2(\Omega)}^{1-\rho},
\end{split}
\end{equation*}
when $0\le t_1<t_2\le T$, $0<T\le 1$.
Finally, making use of  the telescoping series arguments, we  get
\begin{equation*}
\|e^{T\Delta}f\|_{L^2(\Omega)}
\leq Ne^{N/T}\|
\tfrac{\partial}{\partial\nu}\,e^{t\Delta}f
\|_{L^1(\Gamma\times (0,T))}, \ f\in L^2(\Omega).
\end{equation*}
This, together with  Corollary~\ref{cor0}, yields the statement in  Remark \ref{yuanyuan1}.
\end{proof}
\pagebreak



\vfill

\begin{thebibliography}{99}

\bibitem{adolfssonEscauriazakenig} V. Adolfsson, L. Escauriaza, C.E. Kenig. \emph{Convex domains and unique continuation at the boundary.} Rev. Mat. Iberoam. {\bf 11}, 3 (1995) 513--525.

\bibitem{adolfssonEscauriaza} V. Adolfsson, L. Escauriaza. \emph{$C^{1,\alpha}$ domains and unique continuation at the boundary.} Commun. Pur. Appl. Math. L (1997) 935--969.

\bibitem{ApraizEscauriaza1} J.ÊApraiz, L. Escauriaza. \emph{Null-Control and Measurable Sets.} To appear in ESAIM:COCV. DOI: 10.1051/cocv/2012055.



\bibitem{R1} R. M. Brown, \emph{The method of layer potentials for the heat equation in Lipschitz cylinders.}  Amer. J. Math.  {\bf 111} (1989), 339-379.

\bibitem{R2} R. M. Brown, \emph{The initial-Neumann problem for the heat equation in Lipschitz cylinders.} Trans. of the Amer. Math. Soc. {\bf 130}, 1 (1990), 1--52.

\bibitem{Mdicristorondivessella} M. di Cristo, L. Rondi, S. Vessella. \emph{Stability properties of an inverse parabolic problem with unknown boundaries.} Ann. Mat. Pur. Appl. {\bf 185}, 2 (2005), 223--255.

\bibitem{Dahlberg1} B. E. J. Dahlberg, \emph{Estimates of harmonic measure.} Arch. Ration. Mech. An {\bf 65}, 3 (1977), 275--288.

\bibitem{Escauriaza1} L. Escauriaza. \emph{Carleman Inequalities and the Heat Operator.} Duke Math. J. {\bf 104}, 1 (2000) 113-127.

\bibitem{EscauriazaFernandez1} L. Escauriaza, F.J. Fern\'andez. \emph{Unique Continuation for Parabolic Operators.} Ark. Mat. {\bf 41}, 1 (2003) 35-60.

\bibitem{EscauriazaFernandezVessella} L. Escauriaza, F.J. Fern\'andez, S. Vessella. \emph{Doubling Properties of Caloric Functions.} Appl. Anal. {\bf 85},Ê1-3 (2006) 205--223.

\bibitem{EscauriazaSereginSverak} L. Escauriaza, G. Seregin, V. Sver\v{a}k. \emph{Backward Uniqueness for Parabolic Operators.} Arch. Rational Mech. Anal. {\bf 169} (2003) 147--157.

\bibitem{EscauriazaVega} L. Escauriaza, L. Vega. \emph{Carleman Inequalities and the Heat Operator II.} Indiana U. Math. J. {\bf 50}, 3 (2001) 1149--1169.

\bibitem{F1} E. B. Fabes, N. Garofalo, S. Salsa, \emph{A backward Harnack inequality and Fatou theorem for nonnegative solutions of parabolic equation.} Illinois J. Math. {\bf 30} (1986), 536-565.

\bibitem{FabesSalsa1} E. B. Fabes, S. Salsa. \emph{Estimates of caloric measure and the initial-Dirichlet problem for the heat equation in Lipschitz cylinders.} Trans. Amer. Math. Soc. {\bf 279} (1983) 635--650.

\bibitem{Fernandez1} F. J. Fern\'andez. \emph{Unique Continuation for Parabolic Operators II.} Commun. Part. Diff. Eq. {\bf 28}, (2003) 1597--1604.

\bibitem{Fernandez-CaraZuazua1} E. Fern\'andez-Cara, E. Zuazua. \emph{Null and approximate controllability for weakly blowing up semilinear heat equations.} Annales de l'Institut Henri Poincare (C) Non Linear Analysis {\bf 17},Ê5 (2000) 583--616.

\bibitem{Fernandez-CaraZuazua2} E. Fern\'{a}ndez-Cara, E. Zuazua. \emph{The cost of approximate controllability for heat equations: the linear case.} Advances Diff. Eqs. {\bf 5}, 4--6 (2000) 465-514.

\bibitem{FursikovOImanuvilov} A. Fursikov, O.Yu. Imanuvilov. \emph{Controllability of
Evolution Equations.} Seoul National
University, Korea, Lecture Notes Series \textbf{34} (1996).

\bibitem{GarofaloLin} N. Garofalo, F.H. Lin. \emph{Monotonicity properties of variational integrals, $A_p$-weights and
 unique continuation.} Indiana Univ. Math. J. {\bf 35},Ê1-3 (1986) 245--267.

\bibitem{GilbargTrudinger} D. Gilbarg, N.S. Trudinger. \emph{Elliptic Partial Differential Equations of Second Order.} Springer-Verlag (1983).

\bibitem{Hadamard1} J. Hadamard. \emph{Sur les fonctions enti\`{e}res.} Bull. Soc. Math. France {\bf 24} 94-6 (1896).


\bibitem{FJohn2} F. John, \emph{Plane Waves and Spherical Means Applied to Partial Differential Equations.} Interscience Publishers, Inc., New York (1955).

\bibitem{KenigPipher} C.E. Kenig, J. Pipher, \emph{The Neumann problem for elliptic equations with non-smooth coefficients.} Inven. Math. {\bf 113} (1993) 447--509.

\bibitem{kenigwang} C.E. Kenig, W. Wang. \emph{A note on boundary unique continuation for harmonic functions in non-smooth domains.} Potential Analysis {\bf 8}, 2 (1998) 143--147.

\bibitem{KochTataru} H. Koch, D.S. Tataru. \emph{Carleman estimates and unique continuation for second order parabolic equations with non smooth coefficients.} Commun. Part. Diff. Eq. {\bf 34} (2009) 305--366.

\bibitem{KralWendland} J. Kra\'al, W. Wendland. \emph{Some examples concerning applicability of the Fredholm-Radon method in potential theory.} Apl. Mat. {\bf 31}, 4  (1986) 293--308.

\bibitem{KukavicaNystrom} I. Kukavica, K. Nystr\"om. \emph{Unique continuation on the boundary for Dini domains.} Proc. Amer. Math. Soc. {\bf 126}, 2 (1998) 441--446.



\bibitem{G. LebeauL. Robbiano} G. Lebeau, L. Robbiano, \emph{Contr\^ole exact de l'\'equation de la chaleur.}  Commun. Partial Differ. Eqtn.
\textbf{20} (1995) 335-356.

\bibitem{G. LebeauE. Zuazua} G. Lebeau, E. Zuazua. \emph{Null controllability of a system of linear thermoelasticity.}  Arch. Ration. Mech. An. \textbf{141} (4) (1998) 297--329.

\bibitem{Lieberman1} G. M. Lieberman. \emph{Second order parabolic differential equations.} Singapore; World Scientific, 1996.

\bibitem{JLLions} J.L. Lions.\emph{Optimal Control for Systems Governed by Partial Differential Equations.}  Springer-Verlar, Berlin, Heildeberg, New York (1971).


\bibitem{luqi} Q. L\"{u}. \emph{A lower bound on local energy of partial sum of eigenfunctions for Laplace-Beltrami operators.}  To appear in ESAIM: COCV. DOI:10.1051/cocv/2012008.

\bibitem{Malinnikova} E. Malinnikova. \emph{Propagation of smallness for solutions of generalized Cauchy-Riemann systems.} P. Edinburgh Math. Soc.  \textbf{47}  (2004) 191--204.

\bibitem{Miller2} L. Miller. \emph{A direct Lebeau-Robbiano strategy for the observability of heat-like semigroups.} AIMS {\bf 14}, 4 (2010) 1465--1485.

\bibitem{Micu} S. Micu, I. Roventa, M. Tucsnak, \emph{Time optimal boundary controls for the heat equation.} To apper in J. Funct. Anal. (2012), DOI: 10.1016/j.jfa.2012.04.009.

\bibitem{Morrey} C. B. Morrey, \emph{Multiple Integrals in the Calculus of Variations.} Springer (1966).

\bibitem{Nadirashvili2} N. S. Nadirashvili, \emph{A generalization of Hadamard's three circles theorem.} Moscow Univ. Math. Bull. \textbf{31}, 3 (1976) 30--32.

\bibitem{Nadirashvili} N. S. Nadirashvili, \emph{Estimation of the solutions of elliptic equations with analytic coefficients which are bounded on some set.} Moscow Univ. Math. Bull. \textbf{34}, 2 (1979) 44--48.

\bibitem{PhungWang1} K. D. Phung, G. Wang. \emph{An observability estimate for parabolic equations from a measurable set in time and its applications.}  J. Eur. Math. Soc. To appear.

\bibitem{phungwangzhang}   K. D. Phung, G. Wang, X. Zhang. \emph{On the existence of time optimal controls for linear
evolution equations,} Discrete Contin. Dyn. Syst. Ser. B, 8 (2007), pp. 925-941.

\bibitem{RousseauRobbiano2} J. Le Rousseau, G. Lebeau, \emph{On Carleman estimates for elliptic and parabolic operators. Applications to unique continuation and control of parabolic  equations.} ESAIM Control Optim. Calc. Var., doi:10.1051/cocv/2011168


\bibitem{VerchotaVogel} G.C. Verchota, A.L. Vogel, \emph{A Multidirectional Dirichlet Problem.} J. Geom. Anal. \textbf{13}, 3  (2003) 495--520.

\bibitem{Vessella} S. Vessella, \emph{\it A continuous dependence result in the analytic continuation problem.} Forum Math. \textbf{11}, 6  (1999) 695--703.

\bibitem{gengshengwang1} G. Wang. \emph{$L^\infty$-Null Controllability for the Heat Equation and its consequences for the time optimal control problem.} SIAM J. Control Optim. {\bf 47}, 4 (2008) 1701--1720.

\bibitem{canzhang} C. Zhang. \emph{An observability estimate for the heat equation from a product of two measurable sets.} J. Math. Anal. Appl., \textbf{396} (1) (2012), 7-12.

\end{thebibliography}
\end{document}